\documentclass[12pt]{article}
\usepackage{fullpage}
\usepackage{epsf}
\usepackage{epsfig}
\usepackage{amsthm}
\usepackage{amsfonts}
\usepackage{color}
\usepackage{amsmath}
\usepackage{setspace}
\usepackage{hyperref}
\usepackage{algorithmic}
\usepackage[boxed]{algorithm}
\usepackage[usenames, dvipsnames]{xcolor}

\usepackage{graphicx}
\usepackage{caption}
\usepackage{subcaption}
\usepackage[normalem]{ulem}
\usepackage{mathrsfs}

\usepackage{tikz}
\usetikzlibrary{shapes,arrows}

\usepackage{dsfont}
\newcommand{\blind}{0}

\begin{document}

\def\cA{\mathcal{A}}
\def\cB{\mathcal{B}}
\def\cC{\mathcal{C}}
\def\cD{\mathcal{D}}
\def\cE{\mathcal{E}}
\def\cF{\mathcal{F}}
\def\cG{\mathcal{G}}
\def\cH{\mathcal{H}}
\def\cI{\mathcal{I}}
\def\cJ{\mathcal{J}}
\def\cK{\mathcal{K}}
\def\cL{\mathcal{L}}
\def\cM{\mathcal{M}}
\def\cN{\mathcal{N}}
\def\cO{\mathcal{O}}
\def\cP{\mathcal{P}}
\def\cQ{\mathcal{Q}}
\def\cR{\mathcal{R}}
\def\cS{\mathcal{S}}
\def\cT{\mathcal{T}}
\def\cU{\mathcal{U}}
\def\cV{\mathcal{V}}
\def\cW{\mathcal{W}}
\def\cX{\mathcal{X}}
\def\cY{\mathcal{Y}}
\def\cZ{\mathcal{Z}}

\def\bA{\mathbb{A}}
\def\bB{\mathbb{B}}
\def\bC{\mathbb{C}}
\def\bD{\mathbb{D}}
\def\bE{\mathbb{E}}
\def\bF{\mathbb{F}}
\def\bG{\mathbb{G}}
\def\bH{\mathbb{H}}
\def\bI{\mathbb{I}}
\def\bJ{\mathbb{J}}
\def\bK{\mathbb{K}}
\def\bL{\mathbb{L}}
\def\bM{\mathbb{M}}
\def\bN{\mathbb{N}}
\def\bO{\mathbb{O}}
\def\bP{\mathbb{P}}
\def\bQ{\mathbb{Q}}
\def\bR{\mathbb{R}}
\def\bS{\mathbb{S}}
\def\bT{\mathbb{T}}
\def\bU{\mathbb{U}}
\def\bV{\mathbb{V}}
\def\bW{\mathbb{W}}
\def\bX{\mathbb{X}}
\def\bY{\mathbb{Y}}
\def\bZ{\mathbb{Z}}

\def\dE{\mathds{E}}
\def\dM{\mathds{M}}
\def\dP{\mathds{P}}
\def\dR{\mathds{R}}
\def\dQ{\mathds{Q}}

\def\sA{\mathscr{A}}
\def\sB{\mathscr{B}}
\def\sC{\mathscr{C}}
\def\sD{\mathscr{D}}
\def\sE{\mathscr{E}}
\def\sF{\mathscr{F}}
\def\sG{\mathscr{G}}
\def\sH{\mathscr{H}}
\def\sI{\mathscr{I}}
\def\sJ{\mathscr{J}}
\def\sK{\mathscr{K}}
\def\sL{\mathscr{L}}
\def\sM{\mathscr{M}}
\def\sN{\mathscr{N}}
\def\sO{\mathscr{O}}
\def\sP{\mathscr{P}}
\def\sQ{\mathscr{Q}}
\def\sR{\mathscr{R}}
\def\sS{\mathscr{S}}
\def\sT{\mathscr{T}}
\def\sU{\mathscr{U}}
\def\sV{\mathscr{V}}
\def\sW{\mathscr{W}}
\def\sX{\mathscr{X}}
\def\sY{\mathscr{Y}}
\def\sZ{\mathscr{Z}}

\def\mA{\mathsf{A}}
\def\mB{\mathsf{B}}
\def\mC{\mathsf{C}}
\def\mD{\mathsf{D}}
\def\mE{\mathsf{E}}
\def\mF{\mathsf{F}}
\def\mG{\mathsf{G}}
\def\mH{\mathsf{H}}
\def\mI{\mathsf{I}}
\def\mJ{\mathsf{J}}
\def\mK{\mathsf{K}}
\def\mL{\mathsf{L}}
\def\mM{\mathsf{M}}
\def\mN{\mathsf{N}}
\def\mO{\mathsf{O}}
\def\mP{\mathsf{P}}
\def\mQ{\mathsf{Q}}
\def\mR{\mathsf{R}}
\def\mS{\mathsf{S}}
\def\mT{\mathsf{T}}
\def\mU{\mathsf{U}}
\def\mV{\mathsf{V}}
\def\mW{\mathsf{W}}
\def\mX{\mathsf{X}}
\def\mY{\mathsf{Y}}
\def\mZ{\mathsf{Z}}

\newcommand{\pd}[1]{\partial_{#1}}      
\newcommand{\1}{\mathbbm{1}}            

\newtheorem{theorem}{Theorem}
\newtheorem{proposition}[theorem]{Proposition}
\newtheorem{lemma}[theorem]{Lemma}
\newtheorem{corollary}[theorem]{Corollary}

\theoremstyle{definition}
\newtheorem{definition}[theorem]{Definition}
\newtheorem{example}[theorem]{Example}
\theoremstyle{remark}
\newtheorem{remark}[theorem]{Remark}
\newtheorem{question}[theorem]{Question}
\newtheorem{problem}[theorem]{Problem}
\newtheorem{NB}[theorem]{Nota Bene}

\numberwithin{equation}{section}
\numberwithin{theorem}{section}


\if0\blind
{
\title{\vspace{-1.25cm} Nonparametric Adaptive Bayesian Stochastic Control Under Model Uncertainty}
\author{Tao Chen\thanks{Department of Mathematics,
University of Michigan Ann Arbor, East Hall 2859, 530 Church Street, Ann Arbor, MI 48109-1043}
\and
\and Jiyoun Myung\thanks{Department of Statistics and Biostatistics,
California State University, East Bay, North Science 319, Hayward, CA 94542}
}

\date{}

\maketitle
}\fi

\if1\blind
{
  \bigskip
  \bigskip
  \bigskip
  \begin{center}
    {\LARGE\bf Adaptive Bayesian Method and its Application to Stochastic Control}
\end{center}
  \medskip
  \bigskip
} \fi

\begin{abstract}
In this paper we propose a new methodology for solving a discrete time stochastic Markovian control problem under model uncertainty. By utilizing the Dirichlet process, we model the unknown distribution of the underlying stochastic process as a random probability measure and achieve online learning in a Bayesian manner. Our approach integrates optimizing and dynamic learning. When dealing with model uncertainty, the nonparametric framework allows us to avoid model misspecification that usually occurs in other classical control methods. Then, we develop a numerical algorithm to handle the infinitely dimensional state space in this setup and utilizes Gaussian process surrogates to obtain a functional representation of the value function in the Bellman recursion. We also build separate surrogates for optimal control to eliminate repeated optimizations on out-of-sample paths and bring computational speed-ups. Finally, we demonstrate the financial advantages of the nonparametric Bayesian framework compared to parametric approaches such as strong robust and time consistent adaptive.

\vspace{1em}
\end{abstract}


\section{Introduction}
\label{sec:intro}

In solving stochastic control problems, attention has been paid to model risk, the uncertainty about the underlying system dynamics.
As discussed in \cite{Knight1921},  such type of uncertainty must be distinguished from measurable randomness of system realizations, and hereby called Knightian uncertainty.
This ambiguity is expressed either in terms of a parametric family of distributions or a set of probability measures.
In practice, probabilities of interest are often estimated through observing system outcomes.
Then, several families of approaches, such as ``robust'' and ``learning'' methods, are applied to tackle the control problem.

The central idea of robust techniques, which goes back to \cite{GS1989}, is to find the optimal strategy that performs the best in the worst-case scenario.
Hence, a robust stochastic control problem is a form of inf-sup optimization, where supremum is over the family of probabilities and infimum is taken across the control set.
This area has been extensively studied in literature using different approaches, some of which are briefly described in Section~\ref{sec:review_robust}.
There are two issues that must be addressed in these approaches.
First, one usually assumes equal weights for all possible distributions or probability measures in the considered family even when some are much less plausible than the others.
To overcome this drawback, a penalty function of probabilities can be added to the objective function that is going to be optimized.
One challenge in this treatment then becomes how the penalty function is properly chosen.
Second, the set of probabilities is frequently fixed in time even in a dynamic environment, in which 
uncertainty about the system can be reduced as newly incoming information about the underlying system becomes available. To address this issue, the adaptive robust methodology is proposed in \cite{BCCCJ2019}, which initiates the study of dynamically reducing uncertainty of the underlying model while solving robust stochastic control problems.
Serving as a fundamental tool for such an approach, an innovative statistical method of online updating the confidence regions for the unknown system parameters is introduced in \cite{BCC2017}.

Traditionally, one uses the Bayesian method (cf. \cite{Rieder1975, KV2015}) to incorporate learning into solving control problems.
The rationales behind such methods are twofold.
On one hand, the underlying system is learned through observations of the data. On the other hand, by modeling the uncertainty about the true parameters as random variables, posteriors determine the weights to different models.
Hence, the inf-sup formulation is replaced by the weighted average across all possible models, which leads to a ``inf-integral'' problem.
One could naively use the Bayesian technique to learn the system, then subsequently control the learnt system. One concern regarding such implementation for dynamically consistent problems is that the corresponding optimal control is essentially myopic by separating the learning from the control.
As discussed in \cite{KV2015}, the optimal control should ideally be maintained even during the learning phase.
Contrary to the naive Bayesian algorithm that separates the two phases, some other approaches (cf. \cite{Sirbu2014,BCP2016,BCCCJ2019}) account for the controller being cognizant that knowledge of the unknown model may change in the future and therefore, such issues should be addressed at the present time.
Nevertheless, a wholistic Bayesian framework known as Bayesian adaptive control has been explored, in which control and online learning are integrated together.
In Section~\ref{sec:review_bayesian}, we will briefly review such work.

In this study, we propose a nonparametric adaptive Bayesian methodology that solves stochastic control problems under model uncertainty in a discrete time setup according to the Bellman principle.
Some earlier related parametric frameworks are e.g., \cite{Baeuerle2011,KV2015}.
In contrast with these works, our setup does not assume any model for the underlying system process as such. The proposed approach is more data driven and avoids the issue of model misspecification.
In addition, we prove that the optimal control problem satisfies the Bellman principle.
By considering a Borel measurable loss function, we show that optimal selectors exist and are universally measurable with respect to the relevant augmented state variable.
We also use the machine learning technique, namely the Gaussian process surrogates, to numerically solve the Bellman equations. 

The paper is organized as follows.
In Section~\ref{sec:review} we briefly review some of the existing methodologies for solving stochastic control problems subject to model uncertainty, from both robust and Bayesian perspectives. We introduce our nonparametric adaptive Bayesian approach in Section~\ref{sec:method} and present the theoretical results in Section~\ref{sec:sol}.
In Section~\ref{sec:numerics}, as an illustrative example, a utility maximization problem of optimal investment is considered. We solve the problem by utilizing the proposed approach combined with some machine learning techniques.
Finally, we provide a comparative analysis to the existing control methods.

\section{Existing Methodologies for Stochastic Control Problems under Model Uncertainty}
\label{sec:review}

We start our discussion with a review of classical methods and novel approaches introduced recently for solving dynamically consistent stochastic control problems subject to Knightian uncertainty.

\subsection{Robust Methodologies}
\label{sec:review_robust}

In this section, we will mainly discuss the existing parametric robust techniques for dealing with model uncertainty.
Nonetheless, readers should also be aware of other nonparametric robust methods proposed in the past few decades.
In \cite{Cont2006}, the author treats model uncertainty as multiple probability measures and studies its impact on pricing derivatives.
The topic has been receiving more and more attention due to the last financial crisis. The copula approach is partially criticized for the disaster, which the financial industry was using for pricing financial products such as Collateralized Debt Obligations (CDOs) while not accounting for the potential model risk.
Enormous amount of effort has been placed on addressing the issue after the financial market meltdown.
In the breakthrough paper \cite{BouchardNutz2015}, the authors prove a version of the first fundamental theorem of asset pricing under model uncertainty in the quasi-sure sense.
A related work is \cite{BayraktarZhang2016}, which studies the topic by taking into account the transaction costs.
The adaptive robust framework introduced in \cite{BCCCJ2019} incorporates reducing the uncertainty in the robust method, and is applied to time-inconsistent Markovian control problems under model uncertainty in the follow-up work \cite{BCC2020}.

Although nonparametric robust methods are theoretically sound, most of them are difficult to implement in practice due to the optimization required over a family of probability measures.
On the contrary, robust techniques in the parametric setup have been widely used by large banks when addressing the model risk.
By imposing a parametric model with unknown parameters, the numerical part of the work becomes significantly easier as one optimizes over a set of numbers rather than abstract probabilities measures.
To this end, we will go through several important setups of robust stochastic control problems.
In Section~\ref{sec:numerics}, we will also compare our approach to one of the discussed methodologies, strong robust, via an illustrative example.

Let $(\Omega, \cF)$ be a measurable space, and some positive integer $T$ be a fixed time horizon.
Consider a random process $\{Y_t,\ t=0,\ 1,\ \ldots,\ T\}$ taking values in some measurable space. The process $\{Y_t\}$ is assumed to be observed, but its true law is from a family of probability distributions $\{\bP_\theta,\theta\in\mathbf{\Theta}\}$ and corresponds to the unknown parameter $\theta^*$. Denote by $\bF=(\cF_t,\ t=0,\ \ldots,\ T)$ the natural filtration generated by the process $\{Y_t\}$.
A family $\cU$ of $\bF$-adapted processes $\{\varphi_t\}$ that takes values in some measurable space is considered as the set of admissible controls.
Additionally, let $L$ be a function of $Y^0:=\{Y_0,\ldots,Y_T\}$ and $\varphi^0:=\{\varphi_0,\ldots,\varphi_{T-1}\}$.
A stochastic control problem at hand is then formulated as
\begin{align}\label{eq:nouncertainty}
	\inf_{\{\varphi_t\}\in\cU}\dE_{\theta^*}[L(Y^0,\varphi^0)],
\end{align}
given that one knows $\theta^*$.

However, subjected to the Knightain uncertainty, one cannot deal with problem~\eqref{eq:nouncertainty} since the value of $\theta^*$ is unknown. Various robust methodologies are proposed in view of such ambiguity:
\begin{itemize}
\item
\textit{the (static) robust control approach}
\begin{align}\label{eq:static_robust}
	\inf_{\{\varphi_t\}\in\cU}\sup_{\theta\in\mathbf{\Theta}}\bE_\theta[L(Y^0,\varphi^0)],
\end{align}
which optimizes the objective function over the worst-case model through the whole time scale, is discussed in, e.g., \cite{HSTG2006,HS2008,BasarBernhardBook1995}.
\item
\textit{the strong robust control approach}
\begin{align}\label{eq:strong_robust}
	\inf_{\{\varphi_t\}\in\cU}\sup_{\bQ\in\cQ^{\mathbf{K}}}\bE_\bQ[L(Y^0,\varphi^0)],
\end{align}
searches for the worst-case model in each single time period.
Above $\cQ^{\mathbf{K}}$ is a set of probability measures on the canonical space, and $\mathbf{K}$ is the set of sequences of $\{\theta_t\}$ chosen by a Knightian adversary against the controller (cf. \cite{Sirbu2014,BCP2016}).
\item
\textit{the adaptive robust control approach}
\begin{align}\label{eq:adaptive_robust}
	\inf_{\{\varphi_t\}\in\cU}\sup_{\bQ\in\cQ^{\mathbf{\Psi}}}\bE_\bQ[L(Y^0,\varphi^0)],
\end{align}
incorporates learning into the robust methodology by dynamically shrinking the uncertainty set and finds the worst-case model in each time period.
Above $\cQ^{\mathbf{\Psi}}$ is a set of probability measures on the relevant canonical space.
The family $\cQ^{\mathbf{\Psi}}$ is constructed in a way that the set of adversary strategies $\mathbf{\Psi}$ consists of the set-valued processes $\tau(t,\hat{\theta}_t)$ which, for instance, can be chosen as the confidence region of $\theta^*$ at time $t$ based on point estimator $\hat{\theta}_t$. For more details, we refer the readers to \cite{BCCCJ2019}.
\end{itemize}

The classical (but static) robust method is usually conservative by its nature.
As shown in \cite{BCCCJ2019}, for an optimal investment problem that requires the controller to dynamically allocate the wealth in the risk-free asset and a risky asset, the static robust approach will lead to investment in the risk-free asset only through the whole time scale.
As discussed in \cite{LSS2006}, ``If the true model is the worst one, then this solution will be nice and dandy. However, if the true model is the best one or something close to it, this solution could be very bad (that is, the solution need not be robust to model error at all!).''

The strong robust method tries to overcome this drawback by considering the worst-case model at each time period.
While making decisions at time $t$, the controller takes into account that she could change her opinion about the worst one in the future and adapts her strategies to such possibility.
However, as demonstrated in \cite{BCCCJ2019}, the strong robust provides the exact same solution as the static robust approach for certain problems.

A new framework called adaptive robust is proposed recently.
In view of the limitations of the two tactics mentioned above, the adaptive robust method addresses the issue by dynamically updating the parameter space and removing the unlikely models out of consideration.
This procedure is completed by learning about the system dynamics and utilizing the recursively constructed confidence regions for the unknown parameters (cf. \cite{BCC2017}).
When the penalty term is absent in the objective function, it also partially solves the problem of (unreasonably) considering all possible models with equal weights, since some implausible values of the parameters will be removed due to the learning process.
The strong robust approach is essentially a special case of the adaptive robust by fixing the parameter space throughout.
The challenges in scaling such method to high dimensional problems also inspire employment of machine learning techniques to solve robust control problems.
In \cite{CL2019}, the authors propose and develop a novel algorithm for the adaptive robust control by utilizing the ideas from regression Monte Carlo, adaptive experimental design, and statistical learning.
Numerical studies in the paper, as well as in \cite{BCCCJ2019}, show that the adaptive robust achieves a sound balance between being conservative and aggressive.

\subsection{Bayesian Methodologies}
\label{sec:review_bayesian}

As discussed previously, methods of using the Bayesian theory to solve stochastic control problems under model uncertainty have been developed for quite awhile.
The so-called Bayesian adaptive control is studied in various projects, and we refer readers to e.g. \cite{KV2015,Rieder1975} for detailed discussions.
Both references integrate learning (in a Bayesian manner) and optimization, and use the Bellman principle to solve the control problem.
In \cite{KV2015}, sequence of the Bayesian estimators of the unknown parameters constructed via the filtering technique is augmented to the state process,
and the stochastic optimal control problem with partial observations is turned into one with complete observations, which can be solved by dynamic programming.
In \cite{Rieder1975}, the author considers a non-stationary Bayesian dynamic decision model and reduces it to decision models with completely known transition law.
The strategy is the same as in \cite{KV2015}: to augment the set of posterior distributions to the state space.
A similar discussion can also be found in \cite{Baeuerle2011}.
We hereby summarize the corresponding formulation of control problems as follows,
\begin{align}\label{eq:adap_bayesian}
	\inf_{\{\varphi_t\}\in\cU}\int_{\mathbf{\Theta}}\bE_{\theta}[L(Y^0,\varphi^0)]\nu(d\theta),
\end{align}
where $\nu$ is the prior distribution on the parameter space $\mathbf{\Theta}$.
When solving such a problem according to the Bellman principle, intermediate expectations are computed according to the latest posterior distributions.
In a recent work \cite{Cohen2018}, the author considers a filtering problem of discrete time hidden Markov models subject to model uncertainty and uses nonlinear expectations to model the uncertainty.
In particular, the expectation taken under uncertainty is treated as a nonlinear expectation and formulation of the control problem is dynamically consistent so that solutions are obtained by solving Bellman equations.

We want to mention that most of the aforementioned works consider a parametric model for the system process.
This methodology has shortcomings, as enforcing a specific family of distributions for an unknown law is arguably not the optimal starting point.
Nonparametric problems can be considered to address this issue.
In \cite{Ferguson1973,Ferguson1974}, the Dirichlet process is introduced as a prior distribution on the space of probability measures and indeed shown to yield some desirable properties for handling such kind of problems.

A Dirichlet process can be viewed as a probability measure on the set of probability distributions and therefore a random probability measure.
It is shown in \cite{Ferguson1973} that, with respect to the weak convergence topology, the support of a Dirichlet process contains any probability measure whose support is contained in the support of the parameter of the Dirichlet process (cf. Definition~\ref{def:dirichlet} and discussion).
More importantly, the posterior given a sample of observations from the true probability distribution is also a Dirichlet process. Such properties shed light on incorporating dynamic learning into stochastic control problems in a nonparametric way.

Given these desirable properties, a theoretical framework utilizing the Dirichlet process can potentially achieve some success in solving stochastic control problems under model uncertainty.
Indeed, the author in \cite{Ferguson1974} explored using Dirichlet process to handle uncertainty and solve an adaptive investment problem (see \cite[Section~5]{Ferguson1974} for more discussion).
However, a complete and detailed theory of nonparametric Bayesian control, to the best of our knowledge, has not been established.
In this paper, an adaptive Bayesian framework built upon the tools developed for Dirichlet processes is proposed.
We consider a discrete time dynamic stochastic control problem where the noise process is observed but with unknown distribution.
The corresponding formulation of the problem is a blend of online learning and optimal control, for which the Bellman principle and existence of universally measurable selectors are proved.
Our algorithm can also be seen as a new way to construct the augmented state space.
Instead of using posterior distributions as in the existing literature, we recursively update the parameter of the Dirichlet process based on observations of the incoming signal and the resulting sequence is augmented to the state process.
In turn, Borel measurability of the updating rule for the state process is carried out nicely.
This property is essential in our proof for existence of measurable selectors.
Finally, implementation of the approach involves regression/interpolation against measures on the relevant space.
We suggest an approximation of the state space, and a machine learning technique for overcoming the challenge in the numerical example.

There are important drawbacks of the Dirichlet process to be noted. In particular, a sample distribution drawn from the process is discrete with probability one.
Hence, in the resulting inf-integral formulation, the integral is only taken over all discrete probability measures.
It is then worth emphasizing that the proposed methodology is not limited to the Dirichlet process, of which extensions (e.g. \cite{Lo1984}) will apply as well.
In this paper, the Dirichlet process was chosen for the sake of simplicity and illustrative purpose.
Study of nonparametric adaptive Bayesian using random probability measures that sample continuous distributions will be deferred to future work.

\section{Nonparametric Adaptive Bayesian Control Methodology}
\label{sec:method}

In this section, we elaborate on the ideas presented in \cite{Baeuerle2011,KV2015}, and utilize the nonparametric tools introduced in \cite{Ferguson1973,Ferguson1974} to develop a nonparametric Bayesian framework for solving stochastic control problems subjected to Knightian uncertainty.
Towards this end, we begin our presentation with a precise formulation of the problem.

Similar to Section~\ref{sec:review}, fix a finite time horizon $T$ and let $(\Omega,\cF)$ be some measurable space, on which we consider a sequence of $\bR^d$-valued random variables $\{Y_t,\ t=0,\ \ldots,\ T\}$ with its natural filtration $\bF$.
Denote by $U\subset\bR^k$ a compact subset of $\bR^k$.
We assume that there exists an $\bF$-adapted process $\{\varphi_t\}$ which takes values in $U$ and that it plays the role of a control process.
Let $\cU$ be the set of all control processes.
We also consider a noise process $\{Z_t\}$ that is real valued.
For simplicity, we postulate that the sequence is i.i.d.\footnote{We consider $\{Z_t\}$ as an i.i.d. sequence here for illustrative purpose. Of note, our theory also works for more general noise processes as long as a nonparametric Bayesian estimate is feasible.}.
In addition, $\{Z_t\}$ is observed but its true distribution $P_Z$ is unknown.
We describe $\{Y_t\}$ as the state process of some controlled dynamical system, satisfying the following abstract dynamics
\begin{align*}
  Y_{t+1}=f_Y(Y_t,\varphi_t,Z_{t+1}).
\end{align*}
It is further assumed that the function $f_Y:\bR^d\times U\times\bR\to\bR^d$ is continuous.

Denote by $\mathscr{P}(\bR)$ the set of probability measures on $(\bR,\mathscr{B}(\bR))$, where $\mathscr{B}(\bR)$ is the Borel $\sigma$-algebra on $\bR$.
We equip the set $\mathscr{P}(\bR)$ with the Borel $\sigma$-algebra corresponding to the Prokhorov metric.
In this case, continuity of probability measures is equivalent to the weak convergence, and the space $\mathscr{P}(\bR)$ is Polish.
Next, we will recall the definition of Dirichlet process which is the main tool used in this work.

\begin{definition}\label{def:dirichlet}
Let $\alpha$ and $\cD$ be a finite non-null measure and a random probability measure on $(\bR,\mathscr{B}(\bR))$, respectively. We say that $\cD$ is a Dirichlet process with parameter $\alpha$ and write $\cD\in\mathscr{D}(\alpha)$, if for every finite measurable partition $\{B_1,\ldots,B_n\}$ of $\bR$, the random vector $(\cD(B_1),\ldots,\cD(B_n))$ has a Dirichlet distribution with parameter $(\alpha(B_1),\ldots,\alpha(B_n))$.
\end{definition}

It is well-known that the support of $\cD$ with respect to the topology of weak convergence is the set of all distributions on $\bR$ whose supports are contained in the support of $\alpha$.
In this paper, we will always take $\alpha$ as a finite measure with full support.
On the other hand, nonparametric learning of an unknown distribution can be done through a sequence of Dirichlet processes in a Bayesian manner.
To this end, for the unknown distribution $P_Z$, we assign a Dirichlet process $\mathscr{D}(\alpha)$ as its prior distribution.
Let $c_0=\alpha(\bR)$ and $P_0=\alpha/c_0$, we will write that the prior for $P_Z$ is $\mathscr{D}(c_0P)$.

Given the observations $Z_1,\ldots, Z_t$, define the random probability measure
$$
\cP_t=\frac{c_0P_0+\sum_{s=1}^t\delta_{Z_s}}{c_0+t},
$$
where $\delta$ is the Dirac measure, and we know that the posterior for $P_Z$ is Dirichlet process $\sD((c_0+t)\cP_t)$.
Clearly, the sequence of random probability measures $\cP_t$, $t=1,\ldots,T$, can be written in the following recursive way
$$
\cP_t=\frac{(c_0+t-1)\cP_{t-1}+\delta_{Z_t}}{c_0+t}=:f^{c_0}_P(t-1,\cP_{t-1},Z_t), \quad t=1,\ldots,T,
$$
with $\cP_0=P_0$.
In this work, the process $\{\cP_t\}$ will represent the dynamic learning of $P_Z$ as the time-$t$ posterior of $P_Z$ is given as $\sD((c_0+t)\cP_t)$, $t=1,\ldots,T$.

Now, we proceed to formulate the nonparametric adaptive Bayesian control problem.
By adopting a similar idea presented in \cite{Baeuerle2011,KV2015,BCCCJ2019}, we consider the augmented state process $X_t=(Y_t,\cP_t)$, $t=0,\ldots,T$, and the augmented state space
$$
E_X=\bR^d\times\sP(\bR).
$$
In view that both $\bR^d$ and $\sP(\bR)$ are Polish spaces and therefore Borel spaces, the Cartesian product $E_X$ with the product topology is also a Borel space and the Borel $\sigma$-algebra $\cE_X$ coincides with the product $\sigma$-algebra.
The process $\{X_t\}$ has the following dynamics,
$$
X_{t+1} = \mathbf{G}^{c_0}(t,X_t,\varphi_t,Z_{t+1}), \quad t=0,\ldots,T-1,
$$
where $\mathbf{G}^{c_0}$ is defined as
\begin{align}
  \mathbf{G}^{c_0}(t,x,u,z)=\left(f_Y(y,u,z),f^{c_0}_P(t,P,z)\right),
\end{align}
for $x=(y,P)\in E_X$.

Given our assumptions, the process $\{X_t\}$ is $\bF$-adapted and is Markovian.
Therefore, we are essentially dealing with a Markov decision problem.
This leads to the fact that our optimal control at any time $t=0,\ \ldots,\ T-1$, and given any state $x\in E_X$, will be a function of $t$ and $x$.
See Proposition~\ref{pr:selector} and Theorem~\ref{th:bellman} for the justification.
In order to proceed, we present the following technical result regarding the updating rule $\mathbf{G}^{c_0}$ below.

\begin{lemma}\label{lemma:cont}
  For any $t=0,\ldots, T-1$, the mapping $\mathbf{G}^{c_0}(t,\cdot,\cdot,\cdot)$ is continuous.
\end{lemma}

\begin{proof}
  It is enough to show that $f^{c_0}_P(t,P,z):=\frac{(c_0+t)P+\delta_z}{c_0+t+1}$, $P\in\sP(\bR)$, $z\in\bR$, is continuous with respect to $P$ and $z$ for any fixed $t=0,\ldots,T-1$.

  Assume that $(P_n,z_n)\to (P,z)$ where $P,P_n\in\sP(\bR)$, $z,z_n\in\bR$, $n=1,2,\ldots$.
  Then $P_n\to P$ weakly and $z_n\to z$.
  Take $B\subset\bR$ such that
  $$
  \left\{\frac{(c_0+t)P+\delta_z}{c_0+t+1}\right\}(\partial B)=0.
  $$
  Then, set $B$ satisfies that $P(\partial B)=0$ and $z\notin\partial B$.
  According to Portmanteau~theorem, we have $P_n(B)\to P(B)$ and $\delta_{z_n}(B)\to\delta_z(B)$.
  It is implied that
  $$
  \lim_{n\to\infty}\left\{\frac{(c_0+t)P_n+\delta_{z_n}}{c_0+t+1}\right\}(B)=\left\{\frac{(c_0+t)P+\delta_z}{c_0+t+1}\right\}(B).
  $$
  Continuity of $f^{c_0}_P(t,\cdot,\cdot)$ follows by Portmanteau theorem.
\end{proof}

Lemma~\ref{lemma:cont} shows that $\mathbf{G}^{c_0}(t,\cdot,\cdot,\cdot):E_X\times U\times\bR\to E_X$ is a continuous mapping.
As a result, we obtain the following corollary.

\begin{corollary}\label{co:borel}
  The mapping $\mathbf{G}^{c_0}(t,\cdot,\cdot,\cdot)$ is Borel measurable.
\end{corollary}




We proceed to define the transition probability for the process $\{X_t\}$.
That is, for any $t=0,\ldots,T-1$, $(x,u)\in E_X\times U$, we define a probability measure on $\cE_X$ as
\begin{align*}
  Q(B|t,x,u;c_0)&=\int\bP\left(\mathbf{G}^{c_0}(t,x,u,Z_{t+1})\in B\right)\pi(d\bP), \quad \pi\in\mathscr{D}(c_tP),
\end{align*}
for any $B\in\cE_X$.
According to \cite[Proposition 4]{Ferguson1973}, we have that
\begin{align}\label{eq:kernel}
  Q(B|t,x,u;c_0)&=P\left(\mathbf{G}^{c_0}(t,x,u,Z_{t+1})\in B\right), \quad x=(y,P).
\end{align}
It view of Lemma~\ref{lemma:cont}, we have that $Q$ is a Borel measurable stochastic kernel on $E_Y$ given $E_Y\times U$.

\begin{proposition}\label{prop:borel-kernel}
For every fixed $t=0,\ldots,T-1$, $Q(\ \cdot\ |t,\cdot,\cdot\ ;c_0)$ is a continuous stochastic kernel on $E_X$ given $E_X\times U$.
\end{proposition}

\begin{proof}
Take any bounded continuous function $g$ on $E_X$, and consider
\begin{align*}
\int_{E_X}g(x')Q(dx'|t,x,u;c_0)=\int_{\bR}g(\mathbf{G}^{c_0}(t,x,u,z))P(dz).
\end{align*}
From Lemma~\ref{lemma:cont}, $g(\mathbf{G}^{c_0}(t,x,u,z))$ is continuous in $(x,u,z)$.
On the other hand, $P$ can be viewed as a continuous stochastic kernel on $\bR$ given $(x,u)$ since it does not depend on $y$ and $u$.
Then, by \cite[Proposition 7.30]{Bertsekas1978},
$$
\int_{\bR}g(\mathbf{G}^{c_0}(t,x,u,z))P(dz)
$$
is continuous in $(x,u)$, and so is
$$
\int_{E_X}g(x')Q(dx'|t,x,u;c_0).
$$
Hence, Portmanteau theorem implies that $Q$ is continuous in $(x,u)$, which means that $Q$ is a continuous, and obviously Borel measurable stochastic kernel on $E_Y$ given $E_Y\times U$.
\end{proof}

To proceed, we take $\cU$ to be the set of all sequences of universally measurable functions on $E_X$.
Then, for any $c_0>0$, $x_0=\left(y_0,P_0\right)\in E_X$, and control process $\{\varphi_t\}\in\cU$, we denote $\varphi^0=(\varphi_0,\ldots,\varphi_{T-1})$ and define the probability measure $\bQ^{\varphi^0}_{c_0,x_0}$ on the canonical space $E_X^{T+1}$:
\begin{align}\label{eq:canonical_measure1}
  \bQ^{\varphi^0}_{c_0,x_0}(B_0\times\cdots\times B_T)=\int_{B_0}\cdots\int_{B_T}\prod_{t=1}^TQ(dx_t|t-1,x_{t-1},\varphi_{t-1}(x_{t-1});c_0)\delta_{\left(y_0,P_0\right)}(dx_0).
\end{align}
Similarly, we define the probability measure $\bQ_{c_0,x_t}^{\varphi^t}$ on the concatenated canonical space $\mathsf{X}_{s=t+1}^TE_X$ by
\begin{align}\label{eq:canonical_measure2}
  \bQ_{c_0,x_t}^{\varphi^t}(B_{t+1}\times\cdots\times B_T)=\int_{B_{t+1}}\cdots\int_{B_T}\prod_{s=t+1}^TQ(dx_s|s-1,x_{s-1},\varphi_{s-1}(x_{s-1});c_0),
\end{align}
where $\varphi^t:=(\varphi_t,\varphi_{t+1},\ldots,\varphi_{T-1})$, and denote by $\cU^t$ the collection of such sequences.

Now, the \emph{nonparametric adaptive Bayesian control} problem is formulated as
\begin{align}\label{eq:np-bayes-cont}
  \inf_{\{\varphi_t\}\in\cU}\bE_{\bQ^{\varphi^0}_{c_0,x_0}}[\ell(Y_T)],
\end{align}
where $\ell$ is a measurable function.
By employing the canonical construction of the augmented process space $E_X^{T+1}$, dynamic learning of the unknown distribution $P_Z$ is carried out along each path of the process $\{X_t\}$.
In a robust framework such as \cite{BCCCJ2019}, the control problem can be seen as a game between the controller and the Knightian adversary.
The nature maximizes the objective function over the set of probability measures on $E^{T+1}_X$, contrary to the controller's intention to minimize across the admissible strategies.
Therefore, the controller is essentially trying to minimize a nonlinear expectation.
In our formulation, the nature assigns weights to all possible models via a Dirichlet process and chooses her strategy as a weighted average of her options, and in accordance, the controller will minimize a linear expectation on the canonical space.

\begin{remark}
  In this work, we consider the Markov decision problem with terminal loss.
  We want to stress that our framework can be easily extended to deal with problems with intermediate costs.
  We can also adjust the Definitions~\eqref{eq:canonical_measure1}, \eqref{eq:canonical_measure2}, and \eqref{eq:np-bayes-cont} by adopting history-dependent controls.
  Then, it can be applied to non-Markov decision problems.
\end{remark}

\section{Solution to the Adaptive Bayesian Control Problem}
\label{sec:sol}

The main result in this section is to prove that problem~\eqref{eq:np-bayes-cont} satisfies the dynamic programming principle.
Hence, it can be solved recursively, and the optimal control will be obtained.
To this end, we consider the associated adaptive Bayesian Bellman equations
\begin{align}\label{eq:bayes-bellman}
  W^{c_0}_T(x) &= \ell(y), \quad x=(y,P)\in E_X,\nonumber\\
  W^{c_0}_t(x) &= \inf_{u\in U}\bE_{P}\left[W^{c_0}_{t+1}(\mathbf{G}^{c_0}(t,x,u,Z_{t+1}))\right], \quad x=(y,P)\in E_X,\ t=0,\ldots,T-1,
\end{align}
and we will show that
$$
\inf_{\{\varphi_t\}\in\cU}\bE_{\bQ^{\varphi^0}_{c_0,x_0}}[\ell(Y_T)]=W^{c_0}_0(x_0).
$$
To proceed, we will first justify, by using Jankov-von Neumann theorem (\cite[Proposition~7.49 - 7.50]{Bertsekas1978}), that universally measurable selectors $\varphi^*_t(x)$, $t=0,\ldots,T-1$, exist for the associated Bellman equations
  \begin{align}\label{eq:epsilon-optimal}
    \dE_P[W^{c_0}_{t+1}(\mathbf{G}^{c_0}(t,x,\varphi^*_t(x),Z_{t+1}))]=
W^{c_0}_t(x),
  \end{align}
for any $t=0,\ \ldots,\ T-1$.

Towards this end, we postulate that the loss function $\ell:\bR\to\bR$ is Borel measurable.
Then, we have the following result.

\begin{proposition}\label{pr:selector}
  The functions $W^{c_0}_t$, $t=T, T-1,\ldots,0$, are lower semianalytic (l.s.a.), and universally measurable optimal selectors $\varphi^{c_0,*}_t(x)$, $t=T-1,\ldots,0$, in \eqref{eq:bayes-bellman} exist.
\end{proposition}

\begin{proof}
  We will fix $c_0>0$ throughout.
  Note that $\ell$ is Borel measurable, and $\mathbf{G}^{c_0}$ is Borel-measurable according to Lemma~\ref{co:borel}. Thus, $W^{c_0}_T(\mathbf{G}^{c_0}(T-1,\cdot,\cdot,\cdot))$ is Borel measurable and therefore l.s.a. on $E_X\times U$.
  Then, we denote
  $$
  w_{T-1}(x,u)=\dE_{P}\left[W^{c_0}_T(\mathbf{G}^{c_0}(T-1,x,u,Z_T))\right]
  $$
  and Proposition~\ref{prop:borel-kernel} implies that $w_{T-1}$ is an l.s.a. function that maps $E_X\times U$ to $\overline{\bR}$, where $\overline{\bR}$ is the extended real line: $\overline{\bR}=\bR\bigcup\{-\infty,\infty\}$.

  By adopting the notations of \cite[Proposition 7.50]{Bertsekas1978}, we let
  \begin{align*}
    \mathrm{X}&=E_X=\bR\times\sP(\bR), \quad \mathrm{x}=(y,P),\\
    \mathrm{Y}&=U, \quad \mathrm{y}=u,\\
    \mathrm{D}&=E_X\times U,\\
    f(\mathrm{x},\mathrm{y})&=w_{T-1}(y,P,u).
  \end{align*}
  In view of our assumptions, both $\mathrm{X}$ and $\mathrm{Y}$ are Borel spaces. The set $\mathrm{D}$ is Borel and hence analytic. It is trivial to verify that $\text{proj}_\mathrm{X}(D)=E_X$ and $\mathrm{D}_x=U$ for any $x\in E_X$. Define $w^*_{T-1}:E_X\to \overline{\bR}$ by
$$
w^*_{T-1}(x)=\inf_{u\in U}f(\mathrm{x},\mathrm{y}).
$$
Then, Jankov-von Neumann theorem (cf. \cite[Proposition~7.49 - 7.50]{Bertsekas1978}) yields that for any $\varepsilon>0$, there exists an analytically measurable function $\varphi^{*,\varepsilon}_{T-1}: E_X\to U$ satisfying
\begin{align*}
w_{T-1}(x,\varphi^{*,\varepsilon}_{T-1}(x))=
\begin{cases}
w^*_{T-1}(x) +\varepsilon, \quad &\text{if}\ w^*_{T-1}(x)>-\infty,\\
-1/\varepsilon, \quad &\text{if}\ w^*_{T-1}(x)=-\infty.
\end{cases}
\end{align*}
Next, for every positive integer $n$, there exists an analytically measurable function $\varphi^*_n$ such that
\begin{align*}
w_{T-1}(x,\varphi^{*,n}_{T-1}(x))=
\begin{cases}
w^*_{T-1}(x) +\frac{1}{n}, \quad &\text{if}\ w^*_{T-1}(x)>-\infty,\\
-n, \quad &\text{if}\ w^*_{T-1}(x)=-\infty.
\end{cases}
\end{align*}
Since the set $U$ is compact, then for any fixed $x\in E_X$, there is a convergent subsequence $\{\varphi^{*,{n_k}}_{T-1}(x)\}$.
Define $\tilde{\varphi}^{c_0,*}_{T-1}(x)=\lim_{k\to\infty}\varphi^{*,{n_k}}_{T-1}(x)$, and for the fixed $x$ we have $w_{T-1}(x,\tilde{\varphi}^{*,c_0}_{T-1}(x))=w^*_{T-1}(x)$.
Therefore, the set
$$
I = \{x\in E_X\mid \text{for some } u_x\in U,\ w_{T-1}(x,u_x)=w^*_{T-1}(x)\}
$$
coincides with $E_X$.
In view of \cite[Proposition~7.50]{Bertsekas1978} part(b), with slight abuse of notations, there exists a universally measurable function $\varphi^{c_0,*}_{T-1}(x)$ that is the optimal selector.
Moreover, the function $W^{c_0}_{T-1}(x)=w^*_{T-1}(x)$ is l.s.a..
By \cite[Lemma~7.30]{Bertsekas1978}, $W^{c_0}_{T-1}(\mathbf{G}^{c_0}(T-2,\cdot,\cdot,\cdot))$ is l.s.a..
The rest of the proof follows analogously.
\end{proof}

Next, we move on to prove that problem \eqref{eq:np-bayes-cont} can be solved by using the dynamic programming principle. Towards this end, we define the functions
\begin{align*}
    V^{c_0}_t(x,\varphi^t)=&\bE_{\bQ^{\varphi^t}_{c_0,x}}\left[\ell(Y_T)\right], \quad t=0,\ldots,T-1,\\
    V^{c_0,*}_t(x)=&\inf_{\varphi^t\in\cU^t}\bE_{\bQ^{\varphi^t}_{c_0,x}}\left[\ell(Y_T)\right], \quad t=0,\ldots,T-1,\\
    V^{c_0,*}_T(x)=&\ell(y),
\end{align*}
for $x\in E_X$, and $\varphi^t$ which is a sequence of measurable functions.
We provide the following technical result to show the regularity of the functions $V^{c_0}_t$, $t=1,\ldots,T$, so that they can be integrated.
\begin{lemma}\label{lemma:lsa}
For any $t=0,\ldots,T$, and universally measurable sequence $\varphi^t$, the function $V^{c_0}_t(x,\varphi^t)$ is universally measurable.
\end{lemma}
We omit the proof of Lemma~\ref{lemma:lsa} as it follows easily from \cite[Proposition 7.46]{Bertsekas1978}.
Now, with the support of Proposition~\ref{pr:selector} and Lemma~\ref{lemma:lsa}, we present the main result of this section.

\begin{theorem}\label{th:bellman}
  The process $\{\varphi_t^{c_0,*}\}$ constructed from the selectors in Proposition~\ref{pr:selector} is the solution of the nonparametric Bayesian control problem~\eqref{eq:np-bayes-cont}:
  \begin{align}\label{eq:solution}
  \inf_{\{\varphi_t\}\in\cU}\bE_{\bQ^{\varphi^0}_{c_0,x_0}}[\ell(Y_T)]=V^{c_0,*}_0(x_0)=W^{c_0}_0(x_0).
  \end{align}
  Moreover, for any $t=0,\ldots,T-1$, we have
  \begin{align}\label{eq:solution2}
      V^{c_0}_t(x,\varphi^{c_0,*}_t(x))=V^{c_0,*}_t(x)=W^{c_0}_t(x), \quad x\in E_X.
  \end{align}
\end{theorem}

\begin{proof}
  We prove the result via backward induction in $t=T,\ldots,0$:
  $$
  V^{c_0,*}_t(x)=W^{c_0}_t(x), \quad x\in E_X.
  $$
  First, it is clear that $V^{c_0,*}_T(x)=W^{c_0}_T(x)$, for $x\in E_X$.
  Take $t=T-1$, we have
  \begin{align*}
    V^{c_0,*}_{T-1}(x)&=\inf_{\varphi^{T-1}\in \cU^{T-1}}\bE_{\bQ_{c_0,x}^{\varphi^{T-1}}}[\ell(Y_T)]\\
    &=\inf_{u\in U}\bE_{P}[\ell(\mathbf{G}^{c_0}(T-1,x,u,Z_T))]\\
    &=\inf_{u\in U}\bE_{P}[W^{c_0}_T(\mathbf{G}^{c_0}(T-1,x,u,Z_T))]=W^{c_0}_{T-1}(x_{T-1}), \quad x=(y,P)\in E_X.
  \end{align*}
  Hence, the function $V^{c_0,*}_{T-1}$ is l.s.a., and moreover it is universally measurable.
  For $t=T-2,\ldots,1,0$, assume that $V^{c_0,*}_{t+1}$ is l.s.a..
  Given a universally measurable function $\varphi_t$, the stochastic kernel $Q(\cdot\mid t,x,\varphi_t(x);c_0)$ is universally measurable on $E_X$ given $E_X$.
  Therefore, the following integrals
  \begin{align*}
  \int_{E_X}\bE_{\bQ^{\varphi^{t+1}}_{c_0,x'}}[\ell(Y_T)]Q(dx'|t,x,\varphi_t(x);c_0)
  \end{align*}
  and
  \begin{align*}
      \int_{E_X}V^{c_0,*}_{t+1}(x')Q(dx'|t,x,\varphi_t(x);c_0)
  \end{align*}
  are well defined, where the first one is justified by Lemma~\ref{lemma:lsa}.
  We have by induction
  \begin{align*}
    V^{c_0,*}_t(x)&=\inf_{\varphi^t=(\varphi_t,\varphi^{t+1})\in \cU^t}\int_{E_X}\bE_{\bQ^{\varphi^{t+1}}_{c_0,x'}}[\ell(Y_T)]Q(dx'|t,x,\varphi_t(x);c_0)\\
    &\geq\inf_{\varphi^t=(\varphi_t,\varphi^{t+1})\in \cU^t}\int_{E_X}\inf_{\varphi^{t+1}\in \cU^{t+1}}\bE_{\bQ^{\varphi^{t+1}}_{c_0,x'}}[\ell(Y_T)]Q(dx'|t,x,\varphi_t(x);c_0)\\
    &=\inf_{\varphi^t=(\varphi_t,\varphi^{t+1})\in \cU^t}\int_{E_X}V^{c_0,*}_{t+1}(x')Q(dx'|t,x,\varphi_t(x);c_0)\\
    &=\inf_{u\in U}\int_{E_X}W^{c_0}_{t+1}(x')Q(dx'|t,x,u;c_0)=W^{c_0}_t(x).\\
  \end{align*}
  Next, fix $\varepsilon>0$, for any $x\in E_Y$, let $\varphi^{t+1,\varepsilon}$ be the $\varepsilon$-optimal control process at time $t+1$, namely,
  $$
  \bE_{\bQ^{\varphi^{t+1,\varepsilon}}_{c_0,x}}[\ell(Y_T)]\leq \inf_{\varphi^{t+1}\in\cU^{t+1}}\bE_{\bQ_{c_0,x}^{\varphi^{t+1}}}[\ell(Y_T)]+\varepsilon.
  $$
  We know that $\varphi^{t+1,\varepsilon}$ exists as showed in Proposition~\ref{pr:selector}.
  Then, we obtain that
  \begin{align*}
    V^{c_0,*}_t(x)&=\inf_{\varphi^t=(\varphi_t,\varphi^{t+1})\in \cU^{t}}\int_{E_X}\bE_{\bQ^{\varphi^{t+1}}_{c_0,x'}}[\ell(Y_T)]Q(dx'|t,x,\varphi_t(x);c_0)\\
    &\leq\inf_{\varphi^t=(\varphi_t,\varphi^{t+1}\in \cU^t}\int_{E_X}\bE_{\bQ^{\varphi^{t+1,\varepsilon}}_{c_0,x'}}[\ell(Y_T)]Q(dx'|t,x,\varphi_t(x);c_0)\\
    &\leq\inf_{u\in U}\int_{E_X}W^{c_0}_{t+1}(x')Q(dx'|t,x,\varphi_t(x);c_0)+\varepsilon=W^{c_0}_t(x)+\varepsilon.
  \end{align*}
  Because $\varepsilon$ is arbitrary, we conclude that $W^{c_0}_t(x)=V^{c_0}_t(x)$, and such equality holds true for all $t=T-1,\ldots,0$.
  Equality \eqref{eq:solution2} follows immediately.
\end{proof}

\begin{section}{Application: Nonparametric Adaptive Bayesian Utility Maximization}
\label{sec:numerics}

  In this section we demonstrate our method in the context of a dynamic optimal portfolio selection problem.
Consider a market model consists of a risk-free asset with a constant interest rate $r$, and a risky asset $\{S_t\}$ with the corresponding log-return from time $t$ to $t+1$ denoted by $Z_{t+1}=\log(S_{t+1}/S_t)$.
The dynamics of the wealth process $\{Y_t\}$ in the market produced by a self-financing trading strategy is given by
\begin{align}\label{eq:valueprocess}
Y_{t+1}=Y_t(1+r+\varphi_t(e^{Z_{t+1}}-1-r)), \quad t=0,\ldots,T-1,
\end{align}
with initial wealth $Y_0=y_0$.
Hence, the function $f_Y(y,u,z)=y(1+r+u(e^z-1-r))$.
Above $\varphi_t\in U=[0,1]$ is the proportion of the portfolio wealth invested in the risky asset from time $t$ to $t+1$.
In this setup, the wealth process remains non-negative.
We postulate that $Z_t,\ t=1,\ldots,T$, form an i.i.d. sequence of random variables.
Both processes $\{Y_t\}$ and $\{Z_t\}$ are assumed to be observed and in particular, the true distribution of $\{Z_t\}$ is unknown.
Denote by $\bF$ the natural filtration generated by $\{Y_t\}$, and we consider $\{\varphi_t\}$ as an $\bF$-adapted process.
The prior on the distribution of $Z_i$ is chosen to be $\mathscr{D}(c_0A_0)$, where $c_0$ is some positive real number and $A_0$ is a probability measure on $\bR$ with full support.
Consider the loss function in the form: $\ell(x)=\frac{1-x^{1-\eta}}{1-\eta}$ with $\eta>1$.
Note that such function is bounded from below.
Then the adaptive Bayesian control problem at hand is
\begin{align}\label{eq:example1}
  \inf_{\{\varphi_t\}\in\cU}\bE_{\bQ^{\varphi^0}_{c_0,x_0}}\left[\frac{1-Y_T^{1-\eta}}{1-\eta}\right],
\end{align}
where $x_0=(y_0,A_0)$.
Such problem is in fact equivalent to the one as follows
\begin{align}\label{eq:example2}
  \sup_{\{\varphi_t\}\in\cU}\bE_{\bQ^{\varphi^0}_{c_0,x_0}}\left[\frac{Y_T^{1-\eta}-1}{1-\eta}\right].
\end{align}
Therefore, we are maximizing a CRRA utility of the terminal wealth with high risk aversion coefficient.
The corresponding Bellman equations are then written as
\begin{align}
  W^{c_0}_T(x)&=\frac{y^{1-\eta}-1}{1-\eta},\label{eq:example2-1}\\
  W^{c_0}_t(x)&=\sup_{u\in U}\bE_P\left[W^{c_0}_{t+1}(\mathbf{G}^{c_0}(t,x,u,Z_{t+1}))\right],\label{eq:example2-2}\\
  x&=(y,P)\in E_X,\ u\in U,\ t=0,\ \ldots,\ T-1.\nonumber
\end{align}

Note that the main challenge in applying the nonparametric adaptive Bayesian method and solving \eqref{eq:example2-1} - \eqref{eq:example2-2} is that we need to regress against probability measures.
Indeed, when numerically computing $\bE_P\left[W^{c_0}_{t+1}(\mathbf{G}^{c_0}(t,x,u,Z_{t+1}))\right]$, we will estimate its value through Monte Carlo simulations and therefore interpolation/extrapolation is required to evaluate $W_{t+1}(\cdot)$.
In view of such difficulty, the strategy we propose is to regress against the first $m$ moments of the posterior probability measures instead of against the measures themselves (cf. see \ref{sec:ml} below for more discussion).
Practically, we will face a high dimensional optimization problem where traditional grid-based method is extremely inefficient or impossible.
To this end, we will employ the new machine learning algorithm proposed in \cite{CL2019} that has sound scalability and overcomes the challenges in solving our high dimensional stochastic control problem.
To briefly summarize our numerical algorithm: we will employ the regression Monte Carlo (RMC) paradigm and the Gaussian process (GP) surrogates to recursively compute the optimal strategy $\{\varphi_t\}$ backward in time.
The detailed description is presented in the following section.

\subsection{Machine Learning Algorithm}\label{sec:ml}

The main purpose of this section is to propose a numerical solver for \eqref{eq:example2-1} - \eqref{eq:example2-2} in the same spirit of \cite{CL2019}.
We begin with discretization of the state space and we employ the RMC method to create a stochastic (non-gridded) mesh for the underlying state process.
We will also explain how to handle the issue of regressing against probabilities measures along the way.

One difficulty in discretizing the state space is that the state process $\{Y_t\}$ depends on the unknown control which prevents the direct simulation of $\{Y_t\}$ when we apply the RMC paradigm. Hence, we use the idea of control randomization by generating the values $\widetilde{u}^1_t,\ldots,\widetilde{u}^N_t$, $t=0,1,\ldots,T-1$, uniformly in the set $U$ along each of the $N$ sample paths. We also simulate the return process $\{Z_t\}$ according to some sampling measure and obtain $\widetilde{Z}^1_t,\ldots,\widetilde{Z}^N_t$, $t=1,\ldots,T$.
For each path, we choose the initial wealth $\widetilde{y}^i_0$ and the parameter for Dirichlet prior $\alpha=c_0\widetilde{P}^i_0$, $i=1,\ldots,N$.
Here $c_0$ is some postive real number and $\widetilde{P}^i_0$, $i=1,\ldots,N$, are some probability distributions with full support on $\bR$.
Then, along each sample path, the processes $\{Y_t\}$ and $\{A_t\}$ will be updated according to the mapping $\mathbf{G}^{c_0}$ by using the simulated values $\widetilde{u}^i_t$ and $\widetilde{Z}^i_{t+1}$, $i=1,\ldots,N$, $t=0,\ldots,T-1$, and the sample sites $\widetilde{x}^i_t=(\widetilde{y}^i_t, \widetilde{P}^i_t)$, $i=1,\ldots,N$, $t=0,\ldots,T$ will be obtained.

When applying the dynamic programming to solve \eqref{eq:example2-2} at the sampled sites $\widetilde{x}^i_{t-1}$, $i=1,\ldots,N$, $t=1,\ldots,T-1$, we need to approximate the values of functions $W_{t}(\mathbf{G}^{c_0}(t-1,\widetilde{x}^i_{t-1},\cdot,\cdot))$, $i=1,\ldots,N$, $t=1,\ldots,T-1$.
Therefore, a regression model for $W_t$ is needed.
The difficulty in constructing such a model is twofold.
First, part of the state variable $A_t$ is a probability measure which is essentially an infinite dimensional variable.
In the regression, we will approximate the variable by its first $M$ moments and regress against the vector of such moments instead.
To put it in a different way, we approximate the state space $E_X=\{(y,P)\}$ as $\widetilde{E}_X=\{(y,m^1_P,\ldots,m^M_P)\}$ where $m^i_P$ is the $i$th moment of the probability measure $P$.
With slight abuse of notations, we will use $\widetilde{x}^i_t$, $i=1,\ldots,N$, to denote the $N$ sample sites in the space $\widetilde{E}_X$.
Second, although we reduce the dimension of the regression problem from infinity to $1+M$, by implementing this approximation.
It still leads to a significant high dimensional optimization problem.
Hence, a numerical approach with good scalability is crucial for which we will follow the methods of \cite{CL2019} by constructing nonparametric approximations of value functions $W_t$, $t=1,\ \ldots,\ T-1$, via GP surrogates.

To be more specific, we consider a regression model $\widetilde{W}_t$ of $W^{c_0}_t$ such that for any set of inputs the corresponding values of $\widetilde{W}_t$ are jointly normal distributed.
Given training data $(\widetilde x^i_t, W^{c_0}_t(\widetilde x^i_t))$, $i=1,\ \ldots,\ N$, for any $\widetilde x\in \widetilde{E}_X$, the predicted value $\widetilde{W}_t(\widetilde x)$ is computed as
$$
\widetilde{W}_t(\widetilde x)=(k(\widetilde x,\widetilde x^1_t),\ldots,k(\widetilde x,\widetilde x^N_t))[\mathbf{K}+\epsilon^2\mathbf{I}]^{-1}(W_t(\widetilde x^1),\ldots,W_t(\widetilde x^N))^T,
$$
where $\mathbf{I}$ is the $N\times N$ identity matrix and entries of $\mathbf{K}$ has the form $\mathbf{K}_{i,j}=k(\widetilde x^i_t,\widetilde x^j_t)$, $i,\ j=1,\ \ldots,\ N$.
The function $k(\cdot,\cdot)$ is the kernel function of the GP model and in this project, we choose the Matern-5/2 family.
Through estimating the hyperparameters inside $k(\cdot,\cdot)$, we fit the GP surrogate $\widetilde{W}_t$ and use it in \eqref{eq:example2-2} to compute $\widetilde{W}_{t-1}$.
The overall algorithm is as follows:
\begin{enumerate}
  \item (Assume that $W^{c_0}_{t}(\cdot)$ and $\varphi^{c_0,*}_{t}(\cdot)$ are computed at sampled points, and the GP surrogates $\widetilde{W}_{t+1}$ and $\widetilde{\varphi}_{t+1}$ are fitted.)
  \item For time $t$, any $u\in U$ and each of the sample sites $\{\widetilde x^i_{t-1},\ i=1,\ \ldots,\ N\}\subset\widetilde{E}_X$, use Monte Carlo simulation to approximate
      $$
      \bE_{P}[W^{c_0}_{t}(\mathbf{G}^{c_0}(t-1,\widetilde x^i_{t-1},u,Z_{t}))
      $$
      as the following sum
      $$
      \widehat{W}_{t-1}(\widetilde x^i_{t-1},u)\approx \frac{1}{L}\sum_{j=1}^{L}\widetilde{W}_{t-1}(\mathbf{G}^{c_0}(t-1,\widetilde x^i_{t-1},u,Z^j))
      $$
      where $Z^1,\ \ldots,\ Z^L$ are generated from the distribution $P$.
  \item Solve the optimization problem $W^{c_0}_{t-1}(\widetilde x^i_{t-1})=\sup_{u\in U}\widehat{W}_{t-1}(\widetilde x^i_{t-1},u)$ and obtain the maximizer $\varphi^{c_0,*}_t(\widetilde x^i_{t-1})$, $i=1,\ldots,N$.
  \item Fit the GP surrogates for $\widetilde W_{t-1}$ and $\widetilde \varphi^*_{t-1}$ by using the training data $(\widetilde x^i_{t-1},W^{c_0}_{t-1}(\widetilde x^i_{t-1}))$ and $(\widetilde x^i_{t-1},\varphi^{c_0,*}_{t-1}(\widetilde x^i_{t-1}))$, $i=1,\ldots,N$, respectively.
  \item Goto 1: start the next recursion for $t-2$.
\end{enumerate}
To analyze the performance of the optimal control computed from our algorithm, we generate $N'$ out-of-sample paths as follows.
We first simulate the random noise $Z^i_t$, $i=1,\ldots,N'$, $t=1,\ldots,T$ from the sampling measure, and pick $x^i_0\equiv(y_0,A_0)\in E_X$, $i=1,\ldots,N'$.
For each $x^i_t=(y^i_t,A^i_t)$, compute the first $M$ moments $(m^1_{A^i_t},\ldots,m^M_{A^i_t})$ of $A^i_t$, and use the GP surrogate to estimate the corresponding optimal strategy as $\widetilde{\varphi}_t(x^i_t)=\widetilde\varphi^*_t(y^i_t, m^1_{A^i_t},\ldots,m^M_{A^i_t})$.
Then we update the state process according to $x^i_{t+1}=\mathbf{G}^{c_0}(t,x^i_t,\widetilde{\varphi}_t(x^i_t),Z^i_{t+1})$, $i=1,\ \ldots,\ N'$.
Finally, the expected value of the terminal utility is estimated as $\frac{1}{N'}\sum_{i=1}^{N'}\frac{(y^i_T)^{1-\eta}-1}{1-\eta}$.

\subsection{Other Stochastic Control Methodologies under Model Uncertainty}
We will compare the performance of the adaptive Bayesian strategy with the performance of strategies computed from two other classical stochastic control frameworks under model uncertainty: strong robust and time consistent adaptive.

The main purpose of the comparison is to show the advantage of the nonparametric adaptive Bayesian approach to control methods that assume a parametric model for the underlying random noise.
Typically, when an equity investor deals with model uncertainty, she assumes that the log-returns of the underlying stock across the trading periods are i.i.d. normal random variables with unknown mean $\mu$ and variance $\sigma^2$.
Then, to apply the strong robust approach, the invester constructs a confidence region $C:=\tau(t_0,\hat{\mu}_0,\hat{\sigma}^2_0)$ for $(\mu,\sigma^2)$ based on historical observations, where $t_0$ is the sample size of the historical data and $(\hat\mu_0,\hat\sigma^2_0)$ is the estimator of the unknown parameters $(\mu,\sigma^2)$ based on such data.
Next, she computes the optimal strong robust strategies by solving the following Bellman equations:
\begin{align}
W^{\text{sr}}_T(y_T)&=\frac{y_T^{1-\eta}-1}{1-\eta},\nonumber\\
W^{\text{sr}}_t(y_t)&=\sup_{u\in U}\inf_{(\mu,\sigma^2)\in C}\bE_{\mu,\sigma^2}[W^{\text{sr}}_{t+1}(f_Y(y_t,u,Z_{t+1})], \quad t=0,\ldots,T-1,\label{eq:bellman_sr}
\end{align}
where $\bE_{\mu,\sigma^2}$ denotes the expectation computed corresponding to $(\mu,\sigma^2)$.

The idea of the strong robust method is to find the worst-case parameters $(\underline{\mu}_t(y_t,u),\underline{\sigma}^2_t(y_t,u))\in C$ as measurable functions of the state $y_t$ and the trading strategy $u$, then search for the optimal strategy $\varphi^{*,\text{sr}}_t(y_t)$ that performs the best under its corresponding worst-case model. For more details about the strong robust methodology, please refer to the study in \cite{Sirbu2014,BCP2016}.

\begin{table}[ht]
\centering
\renewcommand{\arraystretch}{1.3}
\begin{tabular}{c|ccccccc}
\hline
 \multicolumn{1}{c}{ }& \multicolumn{7}{||c}{$\hat\mu_0=4.615\times10^{-3}$, $\hat\sigma_0=5.609\times10^{-2}$} \\
 \cline{2-8}
 \multicolumn{1}{c}{ }& \multicolumn{5}{||c}{AB} & \multicolumn{1}{|c}{SR} & \multicolumn{1}{|c}{AD} \\
  \multicolumn{1}{c}{ }& \multicolumn{1}{||c}{$c_0=1$} & $c_0=5$ & $c_0=10$ & \multicolumn{1}{c}{$c_0=20$} & \multicolumn{1}{c}{$c_0=30$} & \multicolumn{1}{|c}{ } & \multicolumn{1}{|c}{ }  \\
  \hline
   mean($W$)  & \multicolumn{1}{||c}{1.8037} & 1.8036 & 1.8036  & 1.8035 & 1.8034 & \multicolumn{1}{|c}{1.8020} & \multicolumn{1}{|c}{1.8026}\\
 var($W$)   & \multicolumn{1}{||c}{4.295e-4} & 4.835e-4 & 5.421e-4 & 6.187e-4 & 6.653e-4 & \multicolumn{1}{|c}{4.917e-14} & \multicolumn{1}{|c}{1.162e-3}\\
 $q_{0.30}(W)$  & \multicolumn{1}{||c}{1.7919} & 1.7915 & 1.7891 & 1.7872 & 1.7849 & \multicolumn{1}{|c}{1.8020} & \multicolumn{1}{|c}{1.7841}\\
 $q_{0.90}(W)$  & \multicolumn{1}{||c}{1.8352} & 1.8379 & 1.8405 & 1.8425 & 1.8485 & \multicolumn{1}{|c}{1.8020} & \multicolumn{1}{|c}{1.8483}\\
 $\text{max}(W)$   & \multicolumn{1}{||c}{1.8721} & 1.8704 & 1.8720 & 1.8718 & 1.8656 & \multicolumn{1}{|c}{1.8020} & \multicolumn{1}{|c}{1.8783}\\
 $\text{min}(W)$   & \multicolumn{1}{||c}{1.7711} & 1.7576 & 1.7536 & 1.7632 & 1.7647 & \multicolumn{1}{|c}{1.8020} & \multicolumn{1}{|c}{1.7123}\\
 \hline
\end{tabular}
\bigskip
\caption{Mean, variance, 30\%-quantile, 90\%-quantile, maximum, and minimum of the out-of-sample terminal utility for the AB, SR and AD methods;  Case~1-1.}
\label{table:case1-1}
\end{table}

\begin{table}[ht]
\centering
\renewcommand{\arraystretch}{1.3}
\begin{tabular}{c|ccccccc}
\hline
 \multicolumn{1}{c}{ }& \multicolumn{7}{||c}{$\hat\mu_0=-3.987\times10^{-3}$, $\hat\sigma_0=6.288\times10^{-2}$} \\
 \cline{2-8}
 \multicolumn{1}{c}{ }& \multicolumn{5}{||c}{AB} & \multicolumn{1}{|c}{SR} & \multicolumn{1}{|c}{AD} \\
  \multicolumn{1}{c}{ }& \multicolumn{1}{||c}{$c_0=1$} & $c_0=5$ & $c_0=10$ & \multicolumn{1}{c}{$c_0=20$} & \multicolumn{1}{c}{$c_0=30$} & \multicolumn{1}{|c}{ } & \multicolumn{1}{|c}{ }  \\
  \hline
   mean($W$)  & \multicolumn{1}{||c}{1.8043} & 1.8038 & 1.8041  & 1.8043 & 1.8038 & \multicolumn{1}{|c}{1.8020} & \multicolumn{1}{|c}{1.8016}\\
 var($W$)   & \multicolumn{1}{||c}{4.092e-4} & 3.350e-4 & 3.356e-4 & 2.701e-4 & 1.768e-4 & \multicolumn{1}{|c}{4.917e-14} & \multicolumn{1}{|c}{2.020e-5}\\
 $q_{0.30}(W)$  & \multicolumn{1}{||c}{1.7940} & 1.7952 & 1.7959 & 1.7972 & 1.7981 & \multicolumn{1}{|c}{1.8020} & \multicolumn{1}{|c}{1.8006}\\
 $q_{0.90}(W)$  & \multicolumn{1}{||c}{1.8362} & 1.8334 & 1.8334 & 1.8262 & 1.8256 & \multicolumn{1}{|c}{1.8020} & \multicolumn{1}{|c}{1.8050}\\
 $\text{max}(W)$   & \multicolumn{1}{||c}{1.8702} & 1.8626 & 1.8639 & 1.8598 & 1.8590 & \multicolumn{1}{|c}{1.8020} & \multicolumn{1}{|c}{1.8146}\\
 $\text{min}(W)$   & \multicolumn{1}{||c}{1.7594} & 1.7638 & 1.7574 & 1.7665 & 1.7801 & \multicolumn{1}{|c}{1.8020} & \multicolumn{1}{|c}{1.7715}\\
 \hline
\end{tabular}
\bigskip
\caption{Mean, variance, 30\%-quantile, 90\%-quantile, maximum, and minimum of the out-of-sample terminal utility for the AB, SR and AD methods;  Case~1-2.}
\label{table:case1-2}
\end{table}

Another approach to deal with model uncertainty is the non-robust adaptive approach based on ``learning''.
Loosely speaking, the investor adapts to her latest belief about the unknown parameters, which learned as the point estimator, and takes actions based on such belief.
At any fixed time point, the controller Typically finds the optimal strategy by solving the Bellman equations according to the current parameter estimate across the remaining timeline.
Such treatment bears the drawback that the computed strategy is time inconsistent: the controller knows that she will change her view about the unknown parameter at all future time points, the inevitable future changes are not taken into consideration in the computation of the strategy.

Modifications can be made to obtain a time consistent adaptive control framework.
In the context of our investment problem, one will solve the following Bellman equations
\begin{align}
W^{\text{ad}}_T(y_T,\hat{\mu}_T,\hat{\sigma}^2_T)&=\frac{y_T^{1-\eta}-1}{1-\eta},\nonumber\\
W^{\text{ad}}_t(y_t,\hat{\mu}_t,\hat{\sigma}^2_t)&=\sup_{u\in U}\bE_{\hat{\mu}_t,\hat{\sigma}^2_t}[W^{\text{ad}}_{t+1}(f_Y(y_t,u,Z_{t+1}),f_{\Theta}(t,\hat{\mu}_t,\hat{\sigma}^2_t,Z_{t+1})],\label{eq:bellman_ad}\\
h(t,\hat{\mu}_t,\hat{\sigma}^2_t,Z_{t+1})&=\left(\frac{t\hat{\mu}_t+Z_{t+1}}{t+1},\frac{t(t+1)\hat\sigma^2_t+t(\hat\mu_t-Z_{t+1})^2}{(t+1)^2}\right),\nonumber
\end{align}
for $t=0,\ldots,T-1$. The role of the function $f_{\Theta}$ is to update the estimators $\hat\mu_t$ and $\hat\sigma^2_t$ based on new observation $Z_{t+1}$.
Recall our earlier discussion about the adaptive robust control approach, and note that the above method is a special case of the adaptive robust by replacing the set $\tau(t,\hat\mu_t,\hat\sigma^2_t)$ with the singleton $\{(\hat\mu_t,\hat\sigma^2_t)\}$.

\begin{figure}
\centering
\begin{tabular}{cc}
\includegraphics[page=1,width=0.48\textwidth]{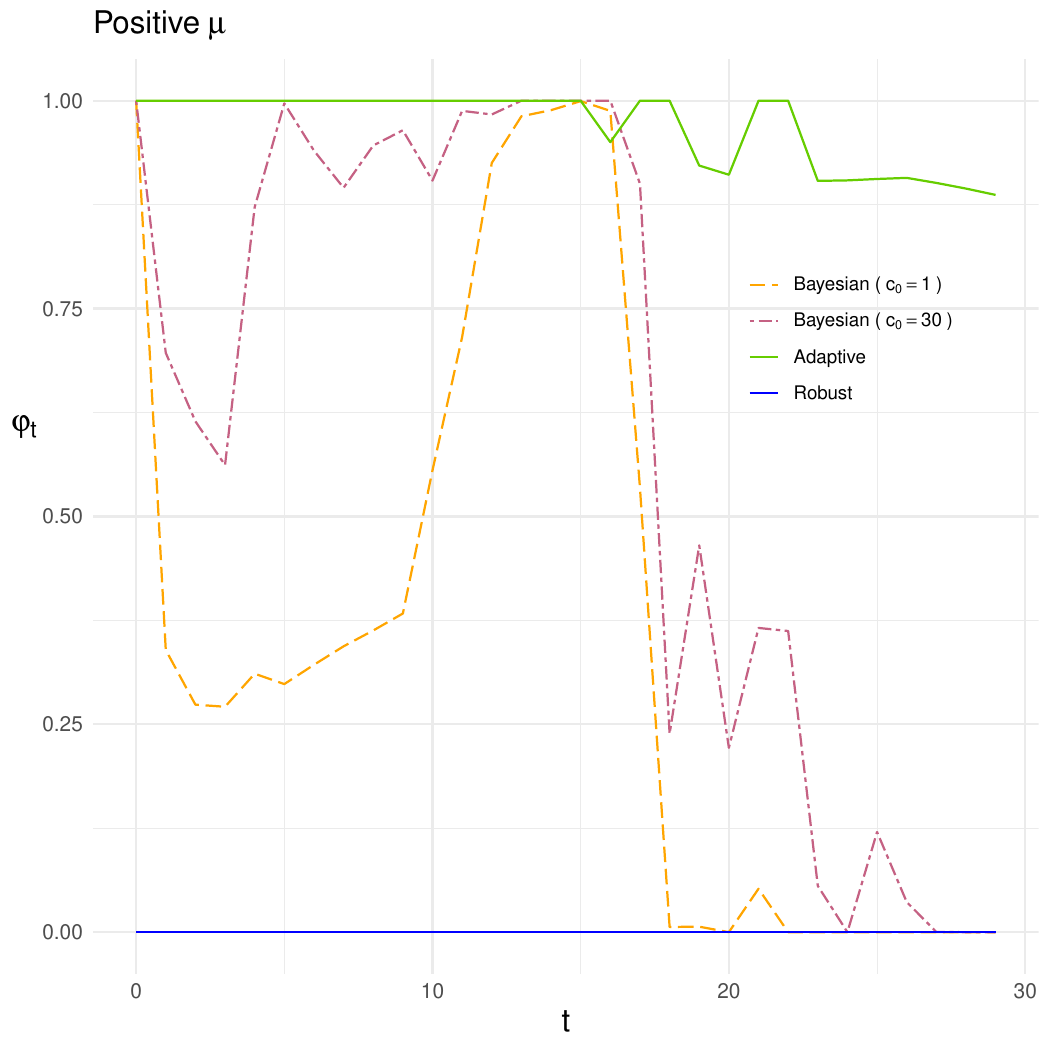} &
\includegraphics[page=3,width=0.48\textwidth]{plots_w_meanbar.pdf} \\
 $\hat\mu_0=4.615\times10^{-3}$, $\hat\sigma_0=5.609\times10^{-2}$ & $\hat\mu_0=-3.987\times10^{-3}$, $\hat\sigma_0=6.288\times10^{-2}$ \\
\end{tabular}
\caption{Path of nonparametric Bayesian strategy $\varphi^{ab}$ in comparison to strong robust and adaptive; Case 1.}
\label{fig:case1}
\end{figure}

\begin{figure}
\centering
\begin{tabular}{cc}
\includegraphics[page=5,width=0.48\textwidth]{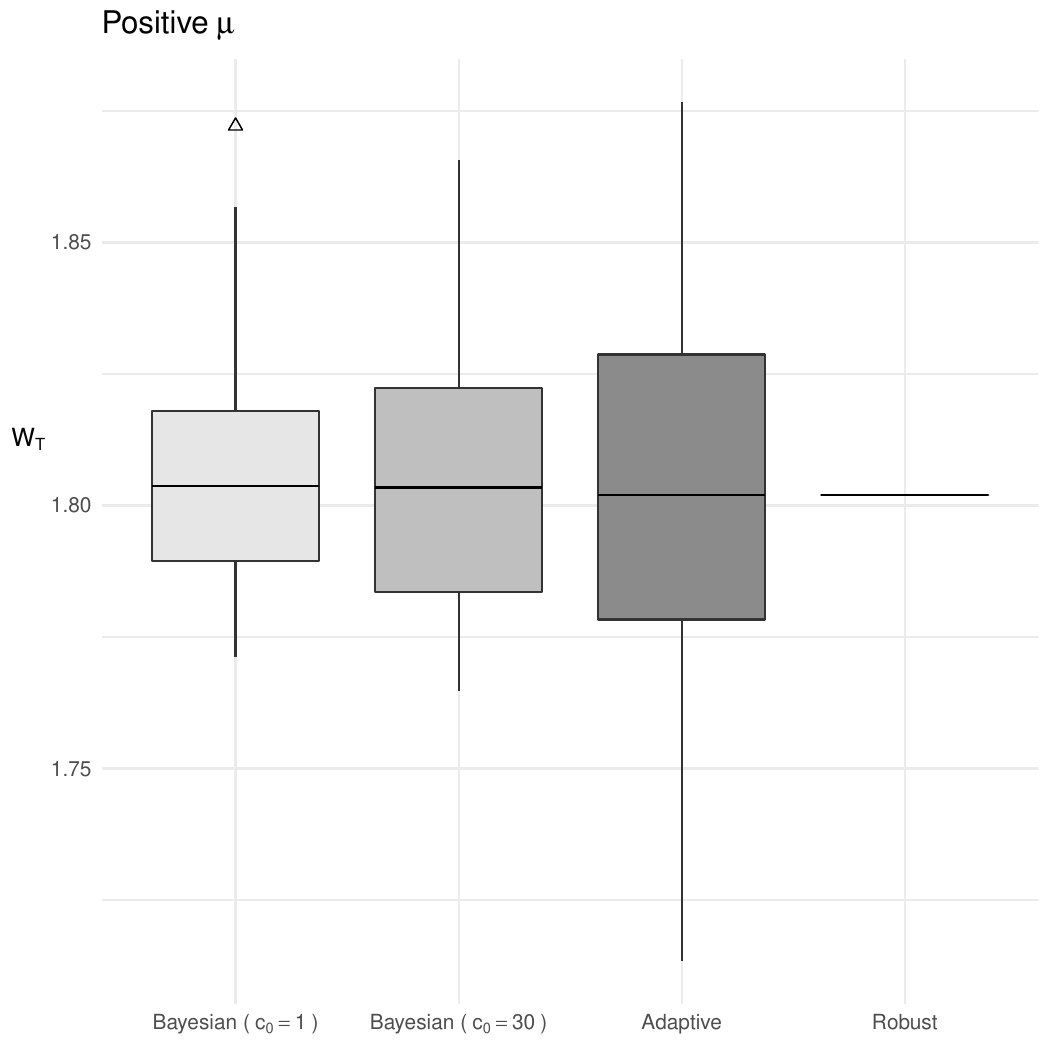} &
\includegraphics[page=7,width=0.48\textwidth]{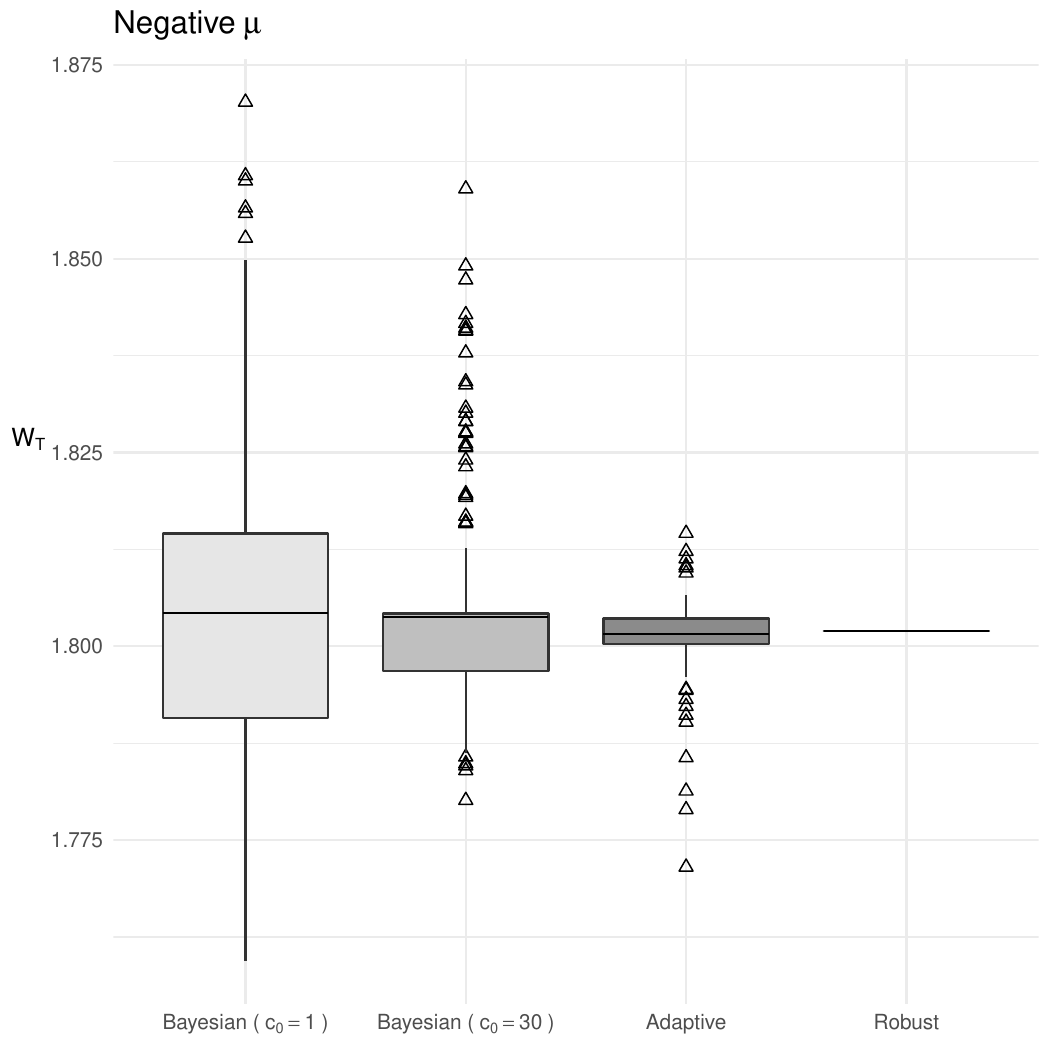} \\
 $\hat\mu_0=4.615\times10^{-3}$, $\hat\sigma_0=5.609\times10^{-2}$ & $\hat\mu_0=-3.987\times10^{-3}$, $\hat\sigma_0=6.288\times10^{-2}$ \\
\end{tabular}
\caption{Distribution of nonparametric Bayesian utility $\varphi^{ab}$ in comparison to strong robust and adaptive; Case 1.}
\label{fig:case1-dist}
\end{figure}

We will omit the detailed description of the algorithms that compute the optimal strong robust and time consistent adaptive strategies, as they are direct modifications of the algorithm introduced in Section~\ref{sec:ml} tailored to Belmman equations \eqref{eq:bellman_sr} and \eqref{eq:bellman_ad}.
For comparison, we will simulate $t_0$ observations $Z_{-t_0},\ldots,Z_{-1}$ from the sampling measure and compute $\hat\mu_0$, $\hat\sigma^2_0$, and $\tau(t_0, \hat\mu_0, \hat\sigma^2_0)$ according to the observations.
The robust parameter set for the strong robust approach is defined as $C=\tau(t_0 , \hat\mu_0, \hat\sigma^2_0)$, and the initial guess of the unknown parameters used in the adaptive approach are chosen as $\hat\mu_0$, and $\hat\sigma^2_0$.
Through the rest of the paper, we denote by $y^{\text{ab}}_T$, $y^{\text{sr}}_T$, and $y^{\text{ad}}_T$ the terminal wealth generated by the adaptive Bayesian, strong robust, and adaptive methods, respectively.
The respective optimal strategies are denoted as $\varphi^{\text{ab},*}_t$, $\varphi^{\text{sr},*}_t$, and $\varphi^{\text{ad},*}_t$, $t=0,\ldots,T-1$.
We will generate $N'$ paths of the out-of-sample random noise $Z^i_t$, $i=1,\ldots,N'$, $t=1,\ldots,T$, from the sampling measure.
Then, we estimate the optimal strategies $\varphi^{\text{ab},*}_t(y^{\text{ab},i}_t)$, $\varphi^{\text{sr},*}_t(y^{\text{sr} ,i}_t)$, and $\varphi^{\text{ad},*}_t(y^{\text{ad} ,i}_t)$, $t=0,\ldots,T-1$, $i=1,\ldots,N'$, by using the corresponding GP surrogates, and in turn update
\begin{align*}
y^{\text{ab},i}_{t+1}&=f_Y(y^{\text{ab},i}_t,\varphi^{\text{ab},*}_t(y^{\text{ab},i}_t),Z^i_{t+1}),\\
y^{\text{sr},i}_{t+1}&=f_Y(y^{\text{sr},i}_t,\varphi^{\text{sr},*}_t(y^{\text{sr},i}_t),Z^i_{t+1}),\\
y^{\text{ad},i}_{t+1}&=f_Y(y^{\text{ad},i}_t,\varphi^{\text{ad},*}_t(y^{\text{ad},i}_t),Z^i_{t+1}),
\end{align*}
for $i=1,\ldots,N'$, and $t=0,\ldots,T-1$.

\subsection{Numerical Results}
In this section, we compare the performance of the adaptive Bayesian, strong robust, and adaptive approaches by analyzing the relevant statistics of
\begin{align*}
W^{\text{ab}}&:=\left(\frac{(y^{\text{ab} ,1}_T)^{1-\eta}-1}{1-\eta},\ldots,\frac{(y^{\text{ab} ,N'}_T)^{1-\eta}-1}{1-\eta}\right),\\
W^{\text{sr}}&:=\left(\frac{(y^{\text{sr} ,1}_T)^{1-\eta}-1}{1-\eta},\ldots,\frac{(y^{\text{sr} ,N'}_T)^{1-\eta}-1}{1-\eta}\right),\\
W^{\text{ad}}&:=\left(\frac{(y^{\text{ad} ,1}_T)^{1-\eta}-1}{1-\eta},\ldots,\frac{(y^{\text{ad} ,N'}_T)^{1-\eta}-1}{1-\eta}\right).
\end{align*}
To this end, we choose one unit of time as $1/30$ year, and $T=30$. The yearly interest rate is 2\%, so that $r=0.02/30=6.667\times10^{-4}$. The number of paths of sample sites for solving the Bellman equations is $N=600$, and the number of out-of-sample paths is $N'=200$. Initial endowment for investing is $y_0=100$. Some other parameters are set as $t_0=100$ and $M=4$.
We consider two cases of bimodal sampling measures.
In case 1, each $Z_t$ with 50\% chance comes from normal distribution $N(\mu_1,\sigma^2_1)$ and with 50\% chance from normal distribution $N(\mu_2,\sigma^2_2)$, where $\mu_1=-0.02/30=-6.667\times10^{-4}$, $\sigma_1=0.4\times\sqrt{1/30}=7.303\times10^{-2}$, and $\mu_2=0.13/30=4.333\times10^{-3}$, $\sigma_2=0.3\times\sqrt{1/30}=5.477\times10^{-2}$.
In case 2, $\mu_1=0.04/30=1.333\times10^{-3}$, $\sigma_1=0.3\times\sqrt{1/30}=5.477\times10^{-2}$, and $\mu_2=0.13/30=4.333\times10^{-3}$, $\sigma_2=0.5\times\sqrt{1/30}=9.129\times10^{-2}$.
The risk-averse parameter $\eta$ chosen for computations in case 1 and 2 are 1.5 and 1.002, respectively.

For both cases, we randomly generate $(\hat\mu_0,\hat\sigma^2_0)$. The initial Dirichlet process for adaptive Bayesian is $\mathscr{D}(c_0 P_0)$ where $P_0$ is normal distribution $N(\hat{\mu}_0,\hat\sigma^2_0)$. A strong robust investor assumes that the one period log-return has a normal distribution $N(\mu,\sigma^2)$ and her robust parameter set $\tau(t_0,\hat{\mu}_0,\hat\sigma^2_0)$ is the 80\% confidence region centered at $(\hat{\mu}_0,\hat\sigma^2_0)$. An adaptive investor also assumes that the model is normal and the initial guess for the parameters are $\hat{\mu}_0$ and $\hat\sigma^2_0$.
\begin{remark}
Note that in the above setup, the adaptive Bayesian investor also assumes that the model for the one-period log-return is $N(\hat{\mu}_0,\hat\sigma^2_0)$ at the starting time.
After that, at any time $t>0$ the model she uses is the weighted average of $N(\hat{\mu}_0,\hat\sigma^2_0)$ and the empirical distribution with respective weights $\frac{c_0}{c_0+t}$ and $\frac{t}{c_0+t}$.
\end{remark}

\begin{table}[ht]
\centering
\renewcommand{\arraystretch}{1.3}
\begin{tabular}{c|ccccccc}
\hline
 \multicolumn{1}{c}{ }& \multicolumn{7}{||c}{$\hat\mu_0=6.255\times10^{-4}$, $\hat\sigma_0=7.090\times10^{-2}$} \\
 \cline{2-8}
 \multicolumn{1}{c}{ }& \multicolumn{5}{||c}{AB} & \multicolumn{1}{|c}{SR} & \multicolumn{1}{|c}{AD} \\
  \multicolumn{1}{c}{ }& \multicolumn{1}{||c}{$c_0=1$} & $c_0=5$ & $c_0=10$ & \multicolumn{1}{c}{$c_0=20$} & \multicolumn{1}{c}{$c_0=30$} & \multicolumn{1}{|c}{ } & \multicolumn{1}{|c}{ }  \\
  \hline
   mean($W$)  & \multicolumn{1}{||c}{4.7079} & 4.7004 & 4.7098  & 4.7096 & 4.7028 & \multicolumn{1}{|c}{4.6038} & \multicolumn{1}{|c}{4.6918}\\
 var($W$)   & \multicolumn{1}{||c}{0.0912} & 0.0674 & 0.0844 & 0.0867 & 0.0850 & \multicolumn{1}{|c}{6.701e-12} & \multicolumn{1}{|c}{0.05942}\\
 $q_{0.30}(W)$  & \multicolumn{1}{||c}{4.5071} & 4.5209 & 4.5130 & 4.5059 & 4.5005 & \multicolumn{1}{|c}{4.6038} & \multicolumn{1}{|c}{4.5507}\\
 $q_{0.90}(W)$  & \multicolumn{1}{||c}{5.1332} & 5.0739 & 5.1463 & 5.1601 & 5.1691 & \multicolumn{1}{|c}{4.6038} & \multicolumn{1}{|c}{5.0224}\\
 $\text{max}(W)$   & \multicolumn{1}{||c}{5.3100} & 5.4023 & 5.5333 & 5.5709 & 5.5200 & \multicolumn{1}{|c}{4.6038} & \multicolumn{1}{|c}{5.310}\\
 $\text{min}(W)$   & \multicolumn{1}{||c}{4.1512} & 4.1635 & 4.2018 & 4.2389 & 4.2327 & \multicolumn{1}{|c}{4.6038} & \multicolumn{1}{|c}{4.0793}\\
 \hline
\end{tabular}
\bigskip
\caption{Mean, variance, 30\%-quantile, 90\%-quantile, maximum, and minimum of the out-of-sample terminal utility for the AB, SR and AD methods;  Case~2-1.}
\label{table:case2-1}
\end{table}

\begin{table}[ht]
\centering
\renewcommand{\arraystretch}{1.3}
\begin{tabular}{c|ccccccc}
\hline
 \multicolumn{1}{c}{ }& \multicolumn{7}{||c}{$\hat\mu_0=-8.347\times10^{-3}$, $\hat\sigma_0=7.805\times10^{-2}$} \\
 \cline{2-8}
 \multicolumn{1}{c}{ }& \multicolumn{5}{||c}{AB} & \multicolumn{1}{|c}{SR} & \multicolumn{1}{|c}{AD} \\
  \multicolumn{1}{c}{ }& \multicolumn{1}{||c}{$c_0=1$} & $c_0=5$ & $c_0=10$ & \multicolumn{1}{c}{$c_0=20$} & \multicolumn{1}{c}{$c_0=30$} & \multicolumn{1}{|c}{ } & \multicolumn{1}{|c}{ }  \\
  \hline
   mean($W$)  & \multicolumn{1}{||c}{4.6980} & 4.6860 & 4.6823  & 4.6668 & 4.6483 & \multicolumn{1}{|c}{4.6038} & \multicolumn{1}{|c}{4.6036}\\
 var($W$)   & \multicolumn{1}{||c}{0.0758} & 0.0640 & 0.0580 & 0.0453 & 0.0310 & \multicolumn{1}{|c}{6.701e-12} & \multicolumn{1}{|c}{6.943e-4}\\
 $q_{0.30}(W)$  & \multicolumn{1}{||c}{4.5139} & 4.5387 & 4.5687 & 4.5929 & 4.5854 & \multicolumn{1}{|c}{4.6038} & \multicolumn{1}{|c}{4.6038}\\
 $q_{0.90}(W)$  & \multicolumn{1}{||c}{5.1097} & 5.0613 & 5.0462 & 5.0009 & 4.8673 & \multicolumn{1}{|c}{4.6038} & \multicolumn{1}{|c}{4.6049}\\
 $\text{max}(W)$   & \multicolumn{1}{||c}{5.4234} & 5.3652 & 5.4676 & 5.4587 & 5.3856 & \multicolumn{1}{|c}{4.6038} & \multicolumn{1}{|c}{4.8089}\\
 $\text{min}(W)$   & \multicolumn{1}{||c}{4.1833} & 4.2164 & 4.2659 & 4.2580 & 4.2897 & \multicolumn{1}{|c}{4.6038} & \multicolumn{1}{|c}{4.3850}\\
 \hline
\end{tabular}
\bigskip
\caption{Mean, variance, 30\%-quantile, 90\%-quantile, maximum, and minimum of the out-of-sample terminal utility for the AB, SR and AD methods;  Case~2-2.}
\label{table:case2-2}
\end{table}

\noindent{\bf Case 1.} In this setup, we randomly generate two sets of values for the initial guess: $(\hat{\mu}_0,\hat\sigma^2_0)=(4.615\times10^{-3},5.609\times10^{-2})$ and $(\hat{\mu}_0,\hat\sigma^2_0)=(-3.987\times10^{-3},6.288\times10^{-2})$.
Then we solve \eqref{eq:example2-2}, \eqref{eq:bellman_sr}, and \eqref{eq:bellman_ad} for these two cases. The resulting strategies from both cases are analyzed on the same set of out-of-sample random noise.

It is worth mentioning that a reasonable choice of the robust parameter set $\tau(t_0,\hat{\mu}_0,\hat\sigma^2_0)$ for the strong robust approach will usually lead to trivial solutions.
Estimating the mean log-return is notoriously inefficient and slow even if the model is indeed Gaussian.
Therefore, by assuming a wrong model in this case, the set $\tau(t_0,\hat{\mu}_0,\hat\sigma^2_0)$ chosen at 80\% confidence level is too large and the worst-case parameter in such set will result a strategy that invests nearly all the money in the banking account at all times (cf. Figure~\ref{fig:case1}).
We also see this effect from both Table~\ref{table:case1-1} and \ref{table:case1-2}, as the mean, quantiles, maximum, and minimum values of $W^{\text{sr}}$ are all the same and the variance of $W^{\text{sr}}$ is almost 0.
Such an extremely conservative strategy will produce a relatively higher 30\% quantile and minimum value of $W^{\text{sr}},$ which are both measures of investment risk.

The adaptive approach chooses the strategy based on the current view of the model parameters which are heavily affected by the initial guess.
Recall that optimal strategies for both cases of $(\hat\mu_0,\hat\sigma^2_0)$ are tested on the same set of out-of-sample random noise.
Two opposite views of the model parameters will lead to strategies that are very different (cf. Figure~\ref{fig:case1}).
For positive $\hat\mu_0$, the AD strategy is very aggressive as we observe much higher values of the mean, 90\% quantile, maximum and much lower values of the 30\% quantile and minimum of $W^{\text{ad}}$ compared to the case of negative $\hat\mu_0$ as in Table~\ref{table:case1-1} and Table~\ref{table:case1-2}.
This means, in general, that the parametric adaptive method is very sensitive to the initial guess and not robust to model mispecification.
Especially when in a market that is neither bull or bear, the investor can easily be confused by the initial guess for relative smaller number of observation and trading periods.

\begin{figure}
\centering
\begin{tabular}{cc}
\includegraphics[page=2,width=0.48\textwidth]{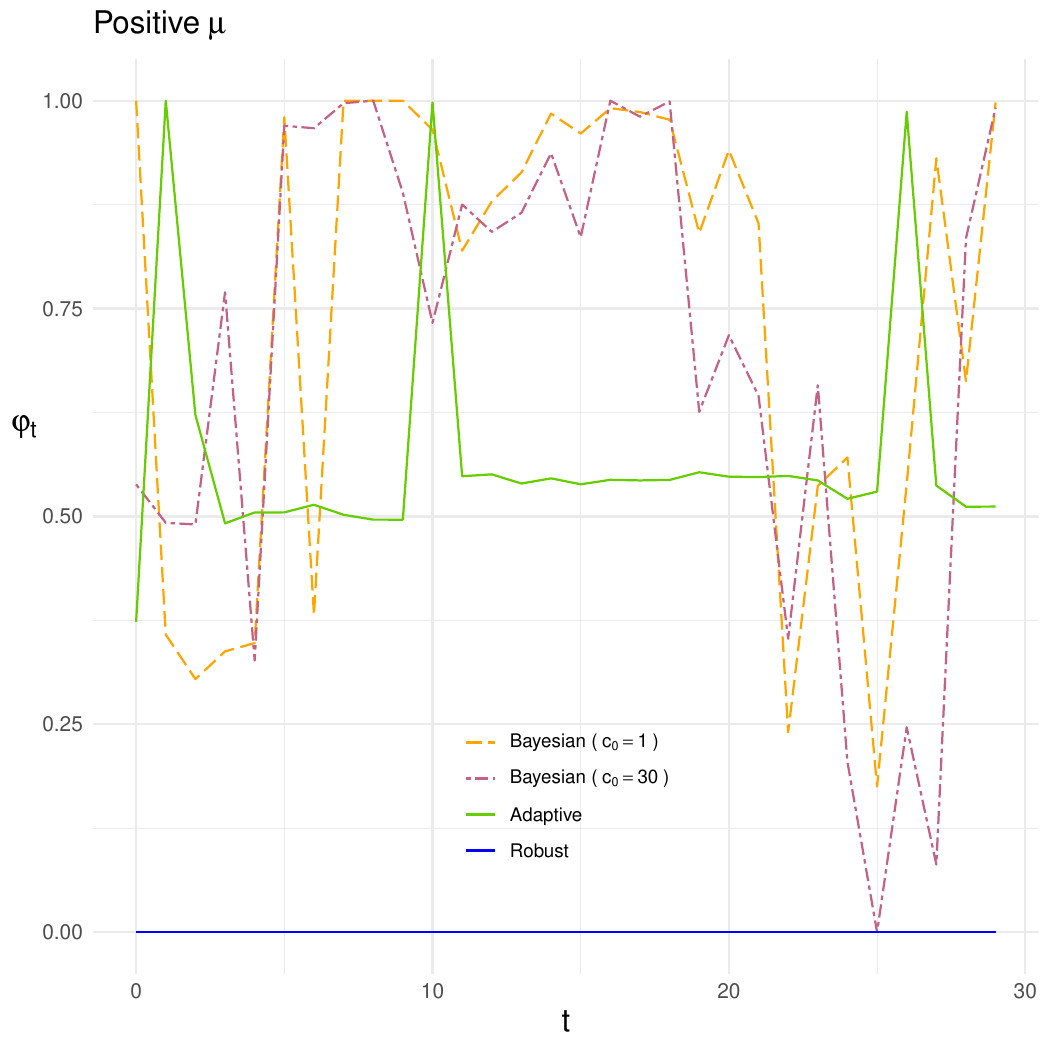} &
\includegraphics[page=4,width=0.48\textwidth]{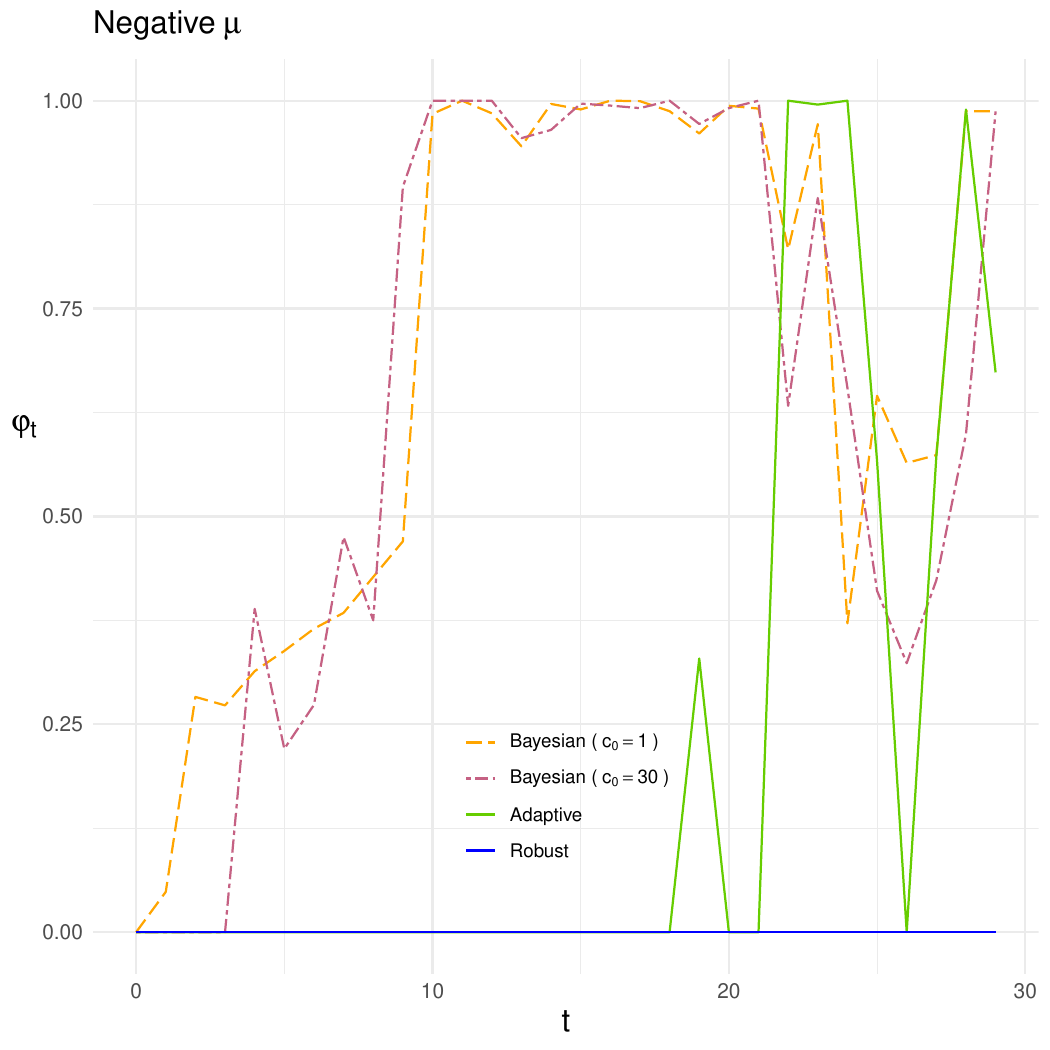} \\
 $\hat\mu_0=6.255\times10^{-4}$, $\hat\sigma_0=7.090\times10^{-2}$ & $\hat\mu_0=-8.347\times10^{-3}$, $\hat\sigma_0=7.805\times10^{-2}$ \\
\end{tabular}
\caption{Path of nonparametric Bayesian strategy $\varphi^{ab}$ in comparison to strong robust and adaptive; Case 2.}
\label{fig:case2}
\end{figure}

\begin{figure}
\centering
\begin{tabular}{cc}
\includegraphics[page=6,width=0.48\textwidth]{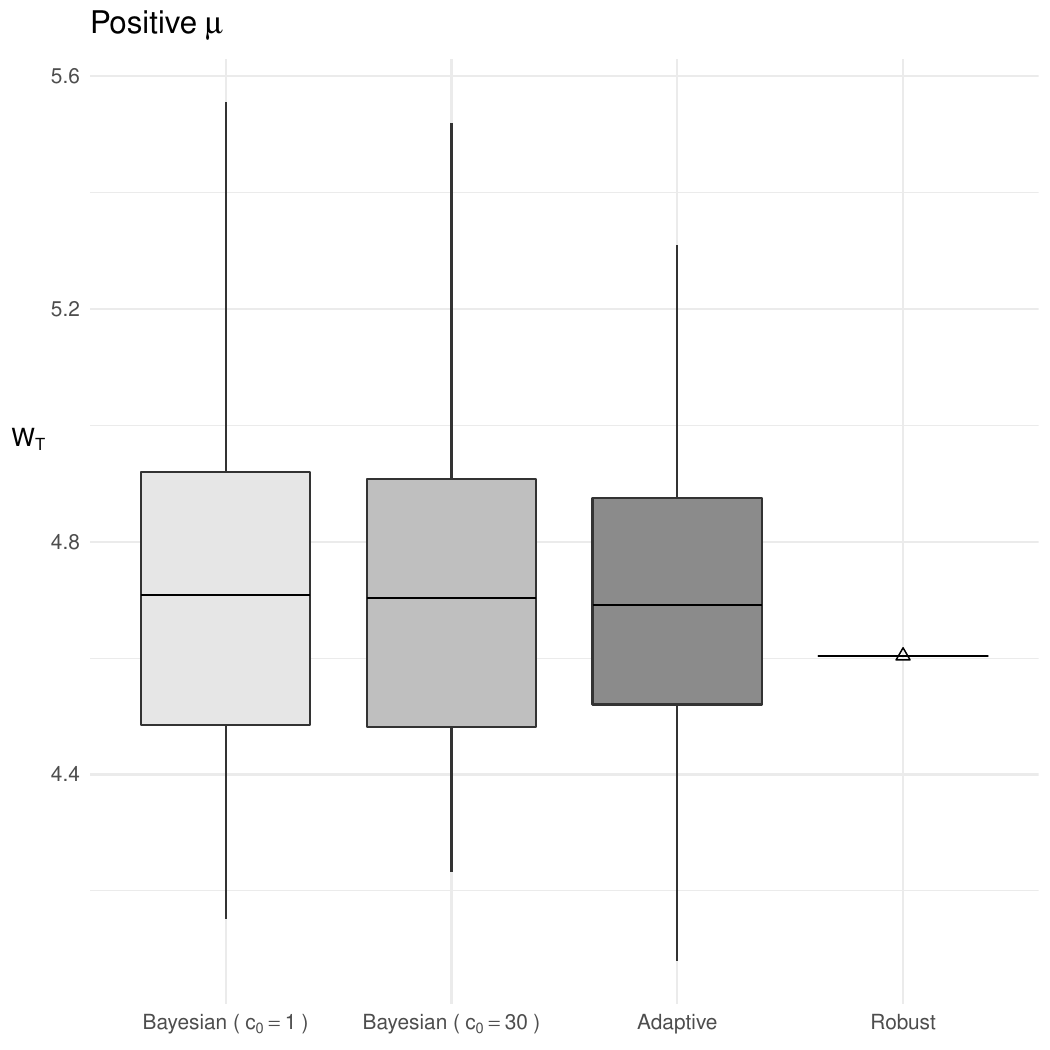} &
\includegraphics[page=8,width=0.48\textwidth]{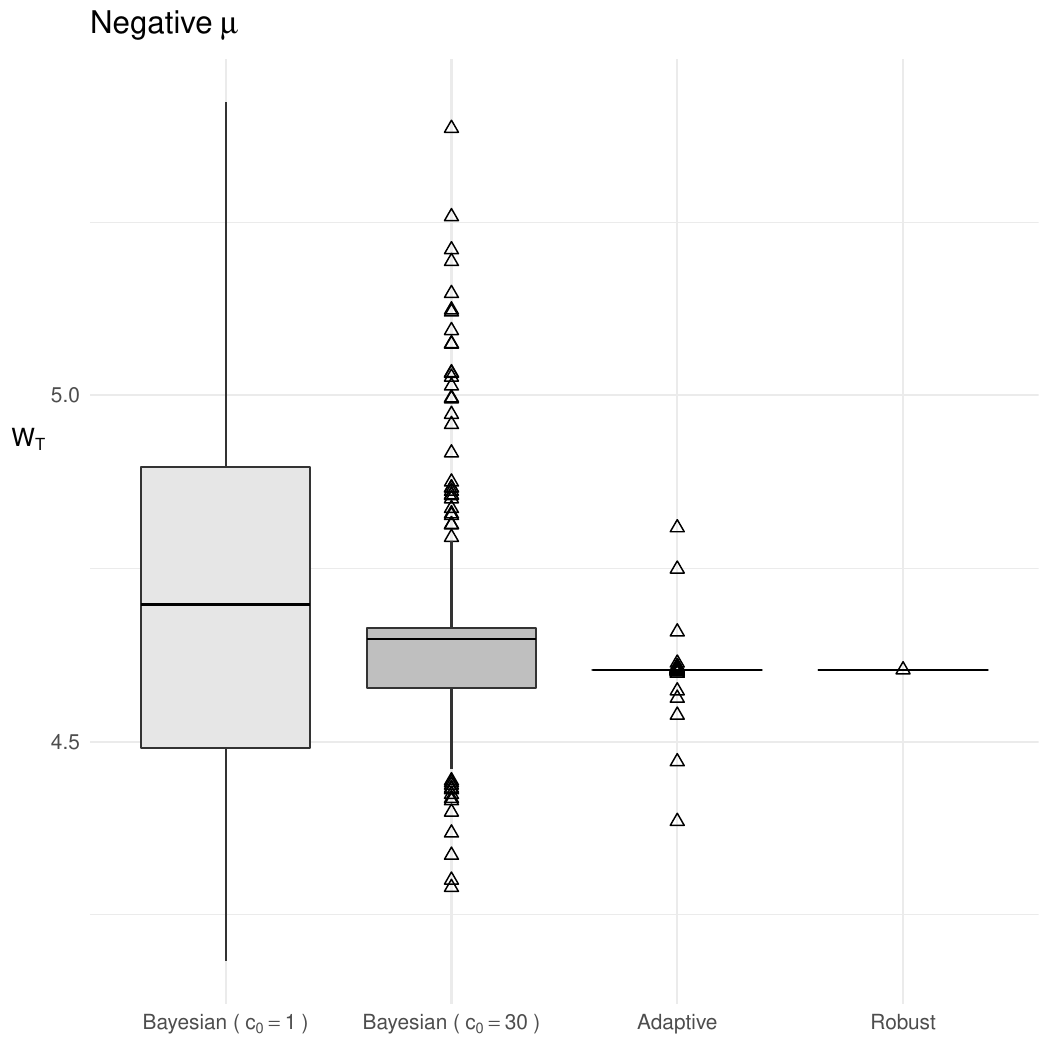} \\
 $\hat\mu_0=6.255\times10^{-4}$, $\hat\sigma_0=7.090\times10^{-2}$ & $\hat\mu_0=-8.347\times10^{-3}$, $\hat\sigma_0=7.805\times10^{-2}$ \\
\end{tabular}
\caption{Distribution of nonparametric Bayesian utility $\varphi^{ab}$ in comparison to strong robust and adaptive; Case 2.}
\label{fig:case2-dist}
\end{figure}

For the adaptive Bayesian framework, we test different choices of $c_0=1$, 5, 10, 20, and 30. As $c_0$ increases, the weight of $P_0$ in the posterior mean of the Dirichlet process increases correspondingly.
Hence the optimal strategy $\varphi^{ab,*}_t$ for $c_0=30$ lies in between $\varphi^{ab,*}_t$ for $c_0=1$ and the AD strategy $\varphi^{ad,*}_t$.
We also observe in Figure~\ref{fig:case1-dist} that the distribution of $W^{ab}$ will converge to that of $W^{ad}$ when $c_0$ becomes large.
Higher weight of $P_0$ will in theory reduce the possibility of overfitting and avoid the learning of the underlying model from picking up too much market noise, especially at early time stages.
It can also be seen as a tuning parameter for the purpose of risk management: in Table~\ref{table:case1-2}, when the initial view of the market is ``pessimistic'' (negative initial guess of the mean log-return), larger $c_0$ will make the investment strategy more conservative and we observe that the 30\% quantile and minimum value of $W^{\text{ab}}$ increases respect to $c_0$.
Accordingly, the 90\% quantile and maximum value decreases since the conservative strategy will be less likely to take advantage of stock price increasing.
In Table~\ref{table:case1-1}, the initial view of the market is ``optimistic'' (positive initial guess of the mean log-return), larger $c_0$ will make the strategy more aggressive and less risk averse.
The variance and 90\% quantile of $W^{\text{ab}}$ becomes larger and the 30\% quantile becomes smaller for higher $c_0$.
Interestingly, the $\max(W^{\text{ad}})$ decreases and $\min(W^{\text{ad}})$ increases in this case.
To understand this, note that in Figure~\ref{fig:case1}, the AD strategy presents a ``mirror'' effect as the strategy decreases when $\hat\mu_0>0$ and increases when $\hat\mu_0<0$ for time steps that are close to $T$.
This means that, by assuming a Gaussian model for the log-return, the AD strategy will converge to a level that is between 0 and 1 on any $Z$-path.
Hence, for large $c_0$, the AB strategy will demonstrate the effect of such convergence.
On the other hand, such observation signals a warning about model misspecification: by mistakenly assuming a Gaussian log-return, the AD strategy for $\hat{\mu}_0<0$ starts to invest money in the risky asset when the market is bad and other investors are reducing their shares of the stock.
In the nonparametric Bayesian framework, a correction will be imposed: in Table~\ref{table:case1-1}, a large enough $c_0$ will eventually make $q_{0.90}(W^{\text{ab}})$ higher than $q_{0.90}(W^{\text{ar}})$; and in Table~\ref{table:case1-2}; it will make $\min(W^{\text{ab}})$ higher than $\min(W^{\text{ar}})$.

\begin{figure}
\centering
\begin{tabular}{cc}
\includegraphics[width=0.48\textwidth]{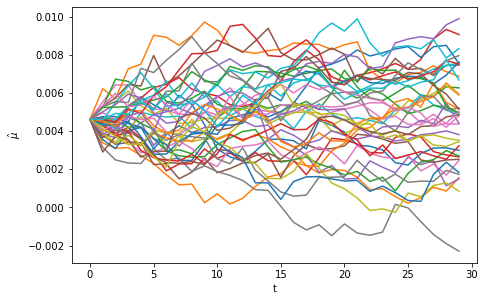} &
\includegraphics[width=0.48\textwidth]{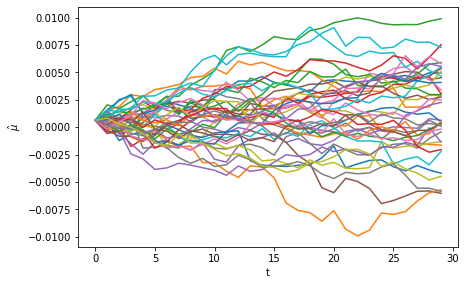} \\
 Case 1 & Case 2 \\
\end{tabular}
\caption{Learning paths of $\hat{\mu}$ by the adaptive approach.}
\label{fig:mu-path}
\end{figure}

Amongst these three methods, it is not surprising that AB produces lower 30\% quantile and minimum value; higher variance, 90\% quantile, and maximum value of the terminal utility compared to the values generated by conservative strategies.
This comparison result is reversed when AB is compared to an aggressive methodology.
Nevertheless, the estimated mean of the terminal utility produced by AB is always higher than the ones given by SR and AD.
These numbers show that AB can be viewed as a preferred approach compared to the other two.
Finally, we stress that in both ``optimistic'' and ``pessimistic'' cases, the numbers generated by AB are similar, which confirms that such methodology is robust to model mispecification and randomness in data observation.

\noindent{\bf Case 2.} In this setup, we randomly generate two sets of values for the initial guess: $(\hat{\mu}_0,\hat\sigma^2_0)=(6.255\times10^{-4}, 7.090\times10^{-2})$ and $(\hat{\mu}_0,\hat\sigma^2_0)=(\hat\mu_0=-8.347\times10^{-3},\hat\sigma_0=7.805\times10^{-2})$.
Then we solve \eqref{eq:example2-2}, \eqref{eq:bellman_sr}, and \eqref{eq:bellman_ad} for these two cases. The resulting strategies from both cases are analyzed on the same set of out-of-sample random noise.

The numerical results we obtain are quite similar to case 1.
One significant difference we have is that for $\hat{\mu}_0>0$, the AD strategy produces a lower $\text{var}(W)$ than AB, and when $c_0$ increases, the variance of the terminal utility from the AB strategy decreases.
This is also confirmed by the box-plot in Figure~\ref{fig:case2-dist}.
In general, a lower variance of the terminal wealth/utility is a result of conservative strategies, and indeed we see from Figure~\ref{fig:case2} that the AD strategy is not necessarily more aggressive than AB in this case.
Such phenomenon is explained by Figure~\ref{fig:mu-path}: in case 2, the estimated $\hat\mu_0$, despite being positive, is smaller than $\hat\mu_0$ in case 1.
Therefore, many paths of $\hat\mu_t$ go below 0 and it is the driving force that the AD strategy becomes conservative even though the initial guess is somewhat optimistic.
In any case, we have that the AB method still produces the highest mean terminal utility, and hence it is the preferred approach compared to the other two.

\end{section}

\begin{section}{Conclusion}
We have developed a nonparametric Bayesian approach to deal with stochastic control problems under model uncertainty.
Our motivation comes from the multiple desirable features of Dirichlet process and aims to avoid model misspecification inherent in assumptions of parametric models.
By augmenting the Bayesian posterior mean to the state variable, we integrate the optimization and online learning when the distribution of the underlying random process is unknown.
We prove the necessary regularity of the relevant functions of the augmented state variable so that the nonparametric adaptive Bayesian control problem is solved by dynamic programming, and the measurable optimal control exists.
The resulting case study provides new insights on the interaction between the prior mean and adaptive learning in the context of utility maximization problem.
The nonparametric framework is robust to the random perturbations in the learning process, and the weight of the prior mean in the dynamic learning can be used as a tuning parameter for the purpose of risk management.

In order to make the proposed framework numerically feasible, we develop an algorithm based on the machine learning technique that utilizes the Gaussian process surrogates.
Following the idea introduced in \cite{CL2019}, we build multiple surrogates for different pieces of the Bellman recursion, not only for the value function but also for the feedback control.
To handle the infinite dimensional state space associated with the nonparametric learning process, we map each distribution to the corresponding vector of moments to reduce the dimension of the state space.
Instead, it is possible to modify the kernel function of the Gaussian process surrogate and enable it to evaluate the distance between probability distributions.
Further investigation of such proposal and the study of extending our approach to the case of multi-dimensional distributions are deferred to future research.
\end{section}

\bibliographystyle{siam}
\bibliography{Bayesian}

\end{document}